\documentclass[11pt]{amsart}
\usepackage{amsmath,latexsym}

\usepackage{color}

\definecolor{Myblue}{rgb}{0.0,0,0.9}
\definecolor{Mygreen}{rgb}{0.2,1,0}

\newtheorem{thm}{Theorem}[section]

\newtheorem{lem}[thm]{Lemma}
\newtheorem{prop}[thm]{Proposition}

\theoremstyle{definition}
\newtheorem{defn}[thm]{Definition}
\theoremstyle{definition}
\newtheorem{exmp}[thm]{Example}
\newtheorem{remark}[thm]{Remark}

\numberwithin{equation}{section}

\newcommand{\benu}{\begin{enumerate}}
\newcommand{\enu}{\end{enumerate}}

\newcommand{\bema}{\left[\begin{array}}
\newcommand{\ema}{\end{array}\right]}

\newcommand{\HVCenter}[1]{\setbox 0=\hbox{#1}%
        \dimen0=\wd0%
        \dimen1=\ht0%
        \divide\dimen0 by 2%
        \divide\dimen1 by 2%
        \hskip -\dimen0%
        \lower \dimen1%
        \box0%
        \hskip -\dimen0}
\newcommand{\HBCenter}[1]{\setbox 0=\hbox{#1}%
        \dimen0=\wd0%
        \dimen1=\ht0%
        \divide\dimen0 by 2%
        \hskip -\dimen0%
        \box0%
        \hskip -\dimen0}
\newcommand{\LTCenter}[1]{\setbox 0=\hbox{#1}%
        \dimen1=\ht0%
        \lower \dimen1%
        \box0%
        \hskip -\dimen0}
\newcommand{\HTCenter}[1]{\setbox 0=\hbox{#1}%
        \dimen0=\wd0%
        \dimen1=\ht0%
        \divide\dimen0 by 2%
        \hskip -\dimen0%
        \lower \dimen1%
        \box0%
        \hskip -\dimen0}
\newcommand{\RTCenter}[1]{\setbox 0=\hbox{#1}%
        \dimen0=\wd0%
        \dimen1=\ht0%
        \hskip -\dimen0%
        \lower \dimen1%
        \box0%
        \hskip -\dimen0}
\newcommand{\RBCenter}[1]{\setbox 0=\hbox{#1}%
        \dimen0=\wd0%
        \dimen1=\ht0%
        \hskip -\dimen0%
        \box0%
        \hskip -\dimen0}
\newcommand{\RVCenter}[1]{\setbox 0=\hbox{#1}%
        \dimen0=\wd0%
        \dimen1=\ht0%
        \divide\dimen1 by 2%
        \hskip -\dimen0%
        \lower \dimen1%
        \box0%
        \hskip -\dimen0}
\newcommand{\LVCenter}[1]{\setbox 0=\hbox{#1}%
        \dimen1=\ht0%
        \divide\dimen1 by 2%
        \lower \dimen1%
        \box0%
        \hskip -\dimen0}

\begin{document}
\sloppy

\title[A criterion for global dimension two  \, ]{
\small {\sc A criterion for global dimension two for strongly \\
simply connected schurian algebras}}

\author[Bordino]{Natalia Bordino}
\address{Natalia Bordino,
Departamento de Matem\'atica, Facultad de Ciencias Exactas y
  Naturales, Funes 3350, Universidad Nacional de Mar del Plata, 7600
  Mar del Plata, Argentina.}
\email{nataliabordino@gmail.com}

\author[Fern\'andez]{Elsa Fern\'andez}
\address{Elsa Fern\'andez,
Facultad de Ingenier\'{\i}a, Universidad Nacional de la
  Patagonia San Juan Bosco, 9120 Puerto Madryn, Argentina.}
\email{elsafer9@gmail.com}

\author[Trepode]{Sonia Trepode}
\address{Sonia Trepode,
Departamento de Matem\'atica, Facultad de Ciencias Exactas y
  Naturales, Funes 3350, Universidad Nacional de Mar del Plata, 7600
  Mar del Plata, Argentina.}
\email{strepode@gmail.com}
\begin{abstract}
The aim of this paper is to provide a criterion to determine, by
quivers with relations, when an  algebra has global dimension at
most two. In order to do that, we introduce a new class of
algebras of global dimension three, and we call them critical
algebras. Furthermore  we give a characterization of critical
algebras by quivers with relations. Our main theorem states that
if a strongly simply connected schurian algebra does not contain a
critical algebra as a full subcategory,\emph{\emph{}} then it has
global dimension at most two.
\end{abstract}

\thanks{This work is part of the Ph.D Thesis of Natalia Bordino, under the supervision of her advisers, Sonia Trepode and Elsa Fern\'andez, presented at Universidad Nacional de Mar del Plata in August of 2011. The third author is a researcher of CONICET, Argentina.}
 \maketitle


\section{Introduction}
The algebras of global dimension two play an important roll in the
representation theory of finite dimensional algebras,  for
example, two remarkable classes of algebras of global dimension
two are tilted and quasitilted algebras. On other hand, in
\cite{A}, Amiot, introduced a new cluster category associated with
an algebra of global dimension two. Hence the interest for
algebras of global dimension two has revived recently in
connection with cluster categories.

An interesting problem is to find a criterion, by quivers with
relations, to determine when an algebra has global dimension two.
Some criterions of this type have been given in the literature, a
criterion was given by Green, Happel and Zacharia, in  \cite{ghz},
for monomial algebras, and  by Igusa and Zacharia, in \cite{iz}, in
the case of incidence algebras.

In this work we provide a criterion for a well known class of
algebras which are quotients of incidence algebras.  We recall that
a subcategory $B$ of $A$ is called \textsl{full} if
$\mbox{Hom}_B(Si, Sj) = \mbox{Hom}_A (Si, Sj)$, for all $i; j \in
(Q_B)_0$. In order to find the desired criterion, we introduce a new
family of algebras and we call them critical algebras. These
algebras have global dimension three and have the property that
every proper full subcategory has global dimension two.  We give
here a characterization of critical algebras in terms of quivers
with relations.

\pagebreak

{\textbf{Proposition.} Let $\Gamma$ be a critical algebra. Then
either $\Gamma$ or $\Gamma^{op}$ is one of the following algebras.

  
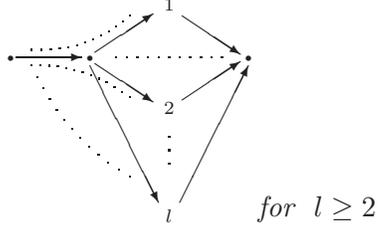
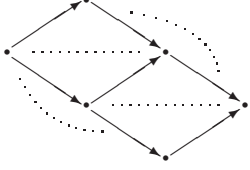
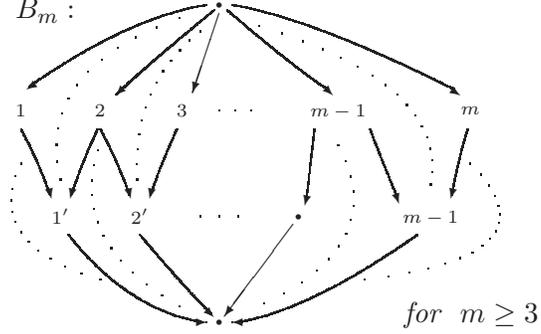
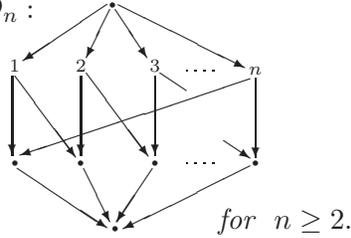
\begin{picture}(10,100)

   \put(5,65){
             \put(0,25){$A_1:$}
      \put(0,0){\circle*{2}}
      \put(30,0){\circle*{2}}
      \put(60,0){\circle*{2}}
      \put(90,0){\circle*{2}}
      \put(2,0){\vector(1,0){26}}
      \put(32,0){\vector(1,0){26}}
      \put(62,0){\vector(1,0){26}}
            \qbezier[16](10,2)(30,20)(45,2)
            \qbezier[16](38,2)(53,20)(72,2)
    }

   \put(215,65){
             \put(0,25){$A_l:$}
             \put(90,-60){ $\mbox{\textit{for }} \: l \geq 2$}
      \put(0,0){\circle*{2}}
      \put(30,0){\circle*{2}}
      \put(60,-19){\HVCenter{\tiny $2$}}
      \put(60,20){\HVCenter{\tiny $1$}}
      \put(90,0){\circle*{2}}
      \put(60,-60){\HVCenter{\tiny $l$}}
      \put(2,0){\vector(1,0){25}}
      \put(32,2){\vector(3,2){22}}
      \put(32,-2){\vector(3,-2){22}}
      \put(30,-3){\vector(1,-2){26}}
      \put(65,16){\vector(3,-2){22}}
      \put(65,-16){\vector(3,2){22}}
      \put(64,-55){\vector(1,2){26}}
            \qbezier[2](60,-40)(60,-35)(60,-30)
            \qbezier[12](8,3)(30,2)(45,16)
            \qbezier[10](40,0)(60,0)(80,0)
            \qbezier[12](8,-3)(28,-2)(45,-15)
            \qbezier[12](8,-3)(20,-30)(45,-45)
    }
\end{picture}

  \begin{picture}(10,130)
   \put(3,70){
             \put(-2,45){$B_1:$}
      \put(0,20){\circle*{2}}
      \put(30,40){\circle*{2}}
      \put(30,0){\circle*{2}}
      \put(60,20){\circle*{2}}
      \put(60,-20){\circle*{2}}
      \put(90,0){\circle*{2}}
      \put(2,22){\vector(3,2){26}}
      \put(2,18){\vector(3,-2){26}}
      \put(32,39){\vector(3,-2){26}}
      \put(32,1){\vector(3,2){26}}
      \put(62,19){\vector(3,-2){26}}
      \put(32,-2){\vector(3,-2){26}}
      \put(62,-19){\vector(3,2){26}}
            \qbezier[12](10,20)(30,20)(50,20)
            \qbezier[12](40,0)(60,0)(80,0)
            \qbezier[12](47,35)(77,30)(80,13)
            \qbezier[12](5,10)(15,-10)(36,-10)
    }

   \put(155,40){
             \put(-2,75){$B_m:$}
             \put(145,-40){$\mbox{\textit{for }} \: m \geq 3$}
      \put(75,80){\circle*{2}}
      \put(0,40){\HVCenter{\tiny $1$}}
      \put(30,40){\HVCenter{\tiny $2$}}
      \put(61,40){\HVCenter{\tiny $3$}}
            \qbezier[2](75,40)(81,40)(87,40)
      \put(120,40){\HVCenter{\tiny $m-1$}}
      \put(170,40){\HVCenter{\tiny $m$}}
      \put(15,0){\HVCenter{\tiny $1'$}}
      \put(45,0){\HVCenter{\tiny $2'$}}
            \qbezier[2](68,0)(75,0)(82,0)
      \put(105,0){\circle*{2}}
      \put(155,0){\HVCenter{\tiny $m-1$}}
      \put(75,-40){\circle*{2}}

            \qbezier(3,49)(40,75)(70,80)
      \put(2.2,48){\vector(-3,-2){0.1}}
            \qbezier(36,46)(45,55)(73,78)
      \put(36,46){\vector(-1,-1){0.1}}
      \put(75,77){\vector(-1,-3){10}}
            \qbezier(116,48)(105,55)(78,78)
      \put(118,46){\vector(1,-1){0.1}}
            \qbezier(164,48)(110,75)(80,80)
      \put(166.2,46){\vector(3,-2){0.1}}

            \qbezier(0,33)(7,20)(11,8)
      \put(11.5,7){\vector(1,-4){0.1}}
            \qbezier(30,33)(37,20)(41,8)
      \put(41.5,7){\vector(1,-4){0.1}}
            \qbezier(132,33)(137,20)(142,8)
      \put(143,5){\vector(1,-3){0.1}}

            \qbezier(29,33)(23,20)(19,8)
      \put(18.7,7){\vector(-1,-4){0.1}}
            \qbezier(59,33)(53,20)(49,8)
      \put(48.7,7){\vector(-1,-4){0.1}}

            \qbezier(111,33)(109,15)(108,7)
      \put(107.5,4.5){\vector(-1,-4){0.1}}

            \qbezier(169,33)(165,20)(163,8)
      \put(162.7,7){\vector(-1,-4){0.1}}

            \qbezier(18,-7)(40,-35)(68,-40)
      \put(70.2,-40){\vector(1,0){0.1}}
            \qbezier(45,-7)(60,-25)(72,-37)
      \put(73,-38){\vector(1,-1){0.1}}
      \put(103,-3){\vector(-3,-4){26}}
            \qbezier(150,-7)(110,-35)(83,-40)
      \put(80,-40){\vector(-1,0){0.1}}

            \qbezier[12](60,75)(5,55)(15,14)
            \qbezier[12](68,70)(42,50)(44,14)
            \qbezier[16](90,72)(162,50)(155,12)

            \qbezier[12](29,25)(18,-10)(60,-34)
            \qbezier[12](120,27)(140,-10)(92,-33)

            \qbezier[12](0,20)(-15,-10)(30,-24)
            \qbezier[16](170,20)(205,-15)(130,-24)

    }

\end{picture}

  \begin{picture}(50,95)
         \put(100,5){
                \put(0,75){$Q_n:$}
                \put(90,-5){$\mbox{\textit{for }} \: n \geq 2$.}
      \put(50,80){\circle*{2}}
      \put(66,57){\HVCenter{\tiny $3$}}
      \put(38,57){\HVCenter{\tiny $2$}}
      \put(13,57){\HVCenter{\tiny $1$}}
      \put(104,55){\HVCenter{\tiny $n$}}
      \put(66,20){\circle*{2}}
      \put(104,20){\circle*{2}}
      \put(13,20){\circle*{2}}
      \put(38,20){\circle*{2}}
      \put(51,-5){\circle*{2}}
      \put(48,80){\vector(-3,-2){32}}
      \put(49,78){\vector(-1,-2){9}}
      \put(51,78){\vector(2,-3){12}}
      \put(52,80){\vector(2,-1){48}}
      \put(66,53){\vector(0,-1){30}}
      \put(38,53){\vector(0,-1){30}}
      \put(12,53){\vector(0,-1){30}}
      \put(104,52){\vector(0,-1){30}}
      \put(102,52){\vector(-3,-1){87}}
      \put(14,18){\vector(3,-2){34}}
      \put(39,18){\vector(1,-2){10}}
            \put(68,54){\line(3,-2){10}}
            \put(102,22){\line(-3,2){10}}
            \put(102,22){\vector(3,-2){0.10}}
      \put(65,18){\vector(-2,-3){13.5}}
      \put(102,19){\vector(-2,-1){48}}
      \put(13,53){\vector(3,-4){23}}
      \put(40,54){\vector(3,-4){23.5}}
             \qbezier[3](78,55)(83,55)(88,55)
             \qbezier[3](78,20)(83,20)(88,20)
    }
\end{picture}

\vskip .5cm

Before stating our main theorem, we need to recall some
definitions. A full subcategory $B$ of $A$ is called
\textit{convex} if any path in $A$ with source and target in $B$
lies entirely in $B$. An algebra $A$ is called \textit{triangular}
if $Q_A$ has no oriented cycles, and it is called
\textit{schurian} if, for all $x, y \in A_0$, we have
$\mbox{dim}_k A(x, y) \leq 1$. A triangular algebra $A$ is called
\textit{simply connected} if, for any presentation $(Q_A, I)$ of
$A$, the group $\pi_1(Q_A, I)$ is trivial, see \cite{BG}. It is
called \textit{strongly simply connected} if every full convex
subcategory of $A$ is simply connected, \cite{Sko}.

Now, we are in a position to state our main theorem.

\textbf{Theorem.} Let $A$ be a strongly simply connected schurian
algebra. Then if $A$ does not contain a critical algebra as a full
subcategory, it follows that $\mbox{gl.dim.}\:A \leq 2$.

The converse of the theorem does not hold.

The paper is organized in the following way.  In Section 2 we
introduce some preliminary concepts and notations. In Section 3 we
construct minimal projective resolutions for the simple modules
over a strongly simply connected schurian algebra and we compute
its projetive dimension. In Section 4 we introduce critical
algebras and we characterize them by quivers with relations.
Finally, we establish the main theorem of this paper.


\section{Preliminaries}

\subsection{Notation}

In this paper, by algebra, we always mean a basic and connected
finite dimensional algebra over an algebraically closed field $k$.
Given a quiver $Q$, we denote by $Q_0$ its set of vertices and by
$Q_1$ its set of arrows. A \textit{relation} in $Q$ from a vertex
$x$ to a vertex $y$ is a linear combination $\rho = \sum^m_{i=1}
\lambda_i w_i$ where, for each $i$, $\lambda_i \in k$ is non-zero
and $w_i$ is a path of length at least two from $x$ to $y$. A
relation in $Q$ is called a \textit{monomial} if it equals a path,
and a \textit{commutativity relation} if it equals the difference
of two paths. A relation $\rho$ is called \textit{minimal} if
whenever $ \rho= \sum_i \beta_i \rho_i\gamma_i $ where $\rho_i$ is
a relation for every $i$, then $\beta_i $ and $\gamma_i$ are
scalars for some index $i$ (see \cite{bmr}).  

We denote by $kQ$ the path algebra of $Q$ and by $kQ(x, y)$ the
$k$-vector space generated by all paths in $Q$ from $x$ to $y$.
For an algebra $A$, we denote by $Q_A$ its quiver. For every
algebra $A$, there exists an ideal $I$ in $kQ_A$, generated by a
set of relations, such that $A \simeq kQ_A/I$ . The pair $(Q_A,
I)$ is called a \textit{presentation} of $A$. An algebra $A =
kQ/I$ can equivalently be considered as a $k$-category of which
the object class $A_0$ is $Q_0$, and the set of morphisms $A(x,
y)$ from $x$ to $y$ is the quotient of $kQ(x, y)$ by the subspace
$I (x, y) = I \cap kQ(x, y)$.

In this work, we always deal with schurian triangular algebras.
For a vertex $x$ in the quiver $Q_A$, we denote by $e_x$ the
corresponding primitive idempotent, $S_x$ the corresponding simple
$A$-module, and by $P_x$ and $I_x$ the corresponding
indecomposable projective and injective A-module, respectively.


Let $Q$ be a connected quiver without oriented cycles. A
\textit{contour} $(p, q)$ in $Q$ from $x$ to $y$ is a pair of
parallel paths of positive length from $x$ to $y$. A contour $(p,
q)$ is called \textit{interlaced} if $p$ and $q$ have a common
vertex besides $x$ and $y$. It is called \textit{irreducible} if
there exists no sequence of paths $p = p_0,p_1, . . . ,p_m = q $
from $x$ to $y$ such that, for each $i$ , the contour
$(p_i,p_{i+1})$ is interlaced.

\subsection{Incidence algebras and their quotients}

Let $(\Sigma,\leq)$ be a finite poset (partially ordered set) with
$n$ elements. The incidence algebra $k\Sigma$ is the subalgebra of
the algebra $M_n(k)$ of all $n\times n$  matrices over $k$
consisting of the matrices $[a_{ij} ]$ satisfying $a_{ij} = 0$ if
$j \not\le i$ . The quiver $Q_{\Sigma}$ of $k\Sigma$ is the
(oriented) Hasse diagram of $\Sigma$, and $k\Sigma \simeq
Q_{\Sigma}/I_{\Sigma}$, where $I_{\Sigma}$ is generated by all
differences $p - q$, with $(p, q)$ a contour in $Q_{\Sigma}$. The
quiver $Q\Sigma$ has no bypass, that is, no subquiver of the form

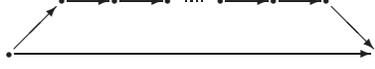
\begin{picture}(0,40)
      \put(110,8){
      \put(0,0){\circle*{2}}
      \put(20,20){\circle*{2}}
      \put(40,20){\circle*{2}}
      \put(60,20){\circle*{2}}
      \put(80,20){\circle*{2}}
      \put(100,20){\circle*{2}}
      \put(120,20){\circle*{2}}
      \put(140,0){\circle*{2}}
      \put(2,2){\vector(1,1){16}}
      \put(22,20){\vector(1,0){16}}
      \put(42,20){\vector(1,0){16}}
            \qbezier[3](67,20)(70,20)(73,20)
      \put(82,20){\vector(1,0){16}}
      \put(102,20){\vector(1,0){16}}
      \put(122,18){\vector(1,-1){16}}
      \put(2,0){\vector(1,0){135}}
    }
\end{picture}

and, conversely, for any quiver $Q$ having no bypass, there exists
a poset $\Sigma$ such that $Q= Q_{\Sigma}$.

If $A$ is an incidence algebra and  $x \in (Q_A)_0 = A_0$, then it
is easy to see that the simple modules  $S_x$, indecomposable
projective $P_x$ and  indecomposable injective $I_x$, are
described, as representation, as follows:

\begin{itemize}
\item $S_x$ is given by $S_x(x)=k$ and $S_x(y)=0$ for $y \neq x$,
plus $S_x(\alpha)=0$, for any arrow $\alpha$.

\item $P_x$ is given by  $P_x(y)=k$ if $x \geq y$ and $P_x(y)=0$ in another case, plus
$P_x(\alpha)=1$ if $x \geq s(\alpha)$ and $P_x(\alpha)=0$
otherwise.

\item $I_x$ is constructed dually to $P_x$.
\end{itemize} \vspace{.1in}

Note that any incidence algebra $A = A(\Sigma)$, the full
subcategory (or full convex) of  $A$ coincide with the incidence
algebras of the full subposet (or full convex) of $\Sigma$.

In \cite{iz} is shown that, if $A = k \Sigma$ is  an incidence
algebra, then  $\mbox{gl.dim.}\: A \leq 2$ iff  $\Sigma$ does not
contain a full subposet isomorphic to $Q_n \; (n \geq 3)$ and all
full subposet of $\Sigma$ isomorphic to $Q_2$ is contained in a
full subposet  $Q$ of $\Sigma$, where

  
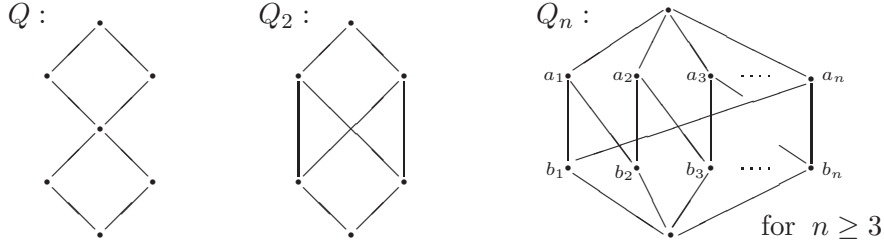
\begin{picture}(50,100)
      \put(25,30){
             \put(-15,60){$Q:$}
      \put(0,40){\circle*{2}}
      \put(20,60){\circle*{2}}
      \put(40,40){\circle*{2}}
      \put(20,20){\circle*{2}}
      \put(0,0){\circle*{2}}
      \put(40,0){\circle*{2}}
      \put(20,-20){\circle*{2}}
      \put(2,42){\line(1,1){16}}
      \put(22,22){\line(1,1){16}}
      \put(2,2){\line(1,1){16}}
      \put(22,-18){\line(1,1){16}}
      \put(2,38){\line(1,-1){16}}
      \put(22,58){\line(1,-1){16}}
      \put(22,18){\line(1,-1){16}}
      \put(2,-2){\line(1,-1){16}}
    }

     \put(120,30){
             \put(-15,60){$Q_2:$}
      \put(0,40){\circle*{2}}
      \put(20,60){\circle*{2}}
      \put(40,40){\circle*{2}}
      \put(0,0){\circle*{2}}
      \put(40,0){\circle*{2}}
      \put(20,-20){\circle*{2}}
      \put(2,42){\line(1,1){16}}
      \put(2,2){\line(1,1){36}}
      \put(22,-18){\line(1,1){16}}
      \put(2,38){\line(1,-1){36}}
      \put(22,58){\line(1,-1){16}}
      \put(2,-2){\line(1,-1){16}}
      \put(0,2){\line(0,1){36}}
      \put(40,2){\line(0,1){36}}
    }

    \put(210,15){
                \put(0,75){$Q_n:$}
                \put(85,-5){$\mbox{for }\: n \geq 3$}
      \put(50,80){\circle*{2}}
      \put(66,55){\circle*{2}}
        \put(65,55){\RVCenter{\tiny $a_3$}}
      \put(38,55){\circle*{2}}
        \put(37,55){\RVCenter{\tiny $a_2$}}
      \put(12,55){\circle*{2}}
        \put(11,55){\RVCenter{\tiny $a_1$}}
      \put(104,54){\circle*{2}}
        \put(108,55){\LVCenter{\tiny $a_n$}}
      \put(66,20){\circle*{2}}
        \put(64,20){\RVCenter{\tiny $b_3$}}
      \put(104,20){\circle*{2}}
        \put(108,20){\LVCenter{\tiny $b_n$}}
      \put(12,20){\circle*{2}}
        \put(11,20){\RVCenter{\tiny $b_1$}}
      \put(38,20){\circle*{2}}
        \put(37,19){\RVCenter{\tiny $b_2$}}
      \put(51,-5){\circle*{2}}
      \put(48,80){\line(-3,-2){34}}
      \put(49,78){\line(-1,-2){10}}
      \put(51,78){\line(2,-3){14}}
      \put(52,80){\line(2,-1){50}}
      \put(66,53){\line(0,-1){30}}
      \put(38,53){\line(0,-1){30}}
      \put(12,53){\line(0,-1){30}}
      \put(104,52){\line(0,-1){30}}
      \put(102,52){\line(-3,-1){87}}
      \put(14,18){\line(3,-2){34}}
      \put(39,18){\line(1,-2){10}}
            \put(68,54){\line(3,-2){10}}
            \put(102,22){\line(-3,2){10}}
      \put(65,18){\line(-2,-3){13.5}}
      \put(102,19){\line(-2,-1){48}}
      \put(13,53){\line(3,-4){23}}
      \put(40,54){\line(3,-4){23.5}}
             \qbezier[3](78,55)(83,55)(88,55)
             \qbezier[3](78,20)(83,20)(88,20)
    }
\end{picture}

We are going to consider quotients of incidence algebras. For such
a quotient $A \simeq kQA/I$ , there exists a poset $\Sigma$ with
$Q_{\Sigma} = Q_A$ and, furthermore, $I = I_{\Sigma} + J$, where
$J$ is an ideal of $kQ_{\Sigma}$ generated by monomials. It is
well known that, if $A$ is schurian strongly simply connected,
then it is a quotient of an incidence algebra, see \cite{d},
\cite{al}.

Conversely, in \cite{acmt}, the authors proved that, if we have a
poset $\Sigma$ such that  $k \Sigma$ is strongly simply connected,
and we consider $A \cong k \Sigma / J$, where $J$ is an ideal of
$k \Sigma$ generated paths are not completely contained in
irreducible contours, then $A$ is strongly simply connected
algebra.

The above results allow us to describe the indecomposable
projective modules of  a strongly simply connected schurian
algebra. Indeed, let $A = kQ_{\Sigma} / I_{\Sigma} + J$ be a
strongly simply connected schurian algebra and let $x \in
(Q_A)_0$. Let $\rho_1, \ldots , \rho_r$ be minimal relations in
$J$ such that $x \geq s(\rho_i)$, for all $1 \leq i \leq r$. Then:

\begin{center}

$P^{A}_x(z)= \left\{ \begin{array}{ll}
k \  \  \  \mbox{if }\  x \geq z, \   z \not\ge t(\rho_i) \   \forall \   1 \leq i \leq r, \\  
0 \  \  \  \mathrm{otherwise},
\end{array} \right.$
\end{center}

\noindent with the induced morphisms.


\section{On the projective dimension of simple modules}

In this Section we consider strongly simply connected schurian
algebras. Our main objective is to describe the first terms of the
minimal projective resolution of a simple module. As a consequence
we are able to study the projective dimension of the simple
modules.

The following remark is important for our purposes.

\begin{remark}\label{rem:rp1}
Let  $A=kQ/I$ be a  strongly simply connected schurian algebra.
Let

\begin{equation}\label{rp1}
  \begin{picture}(80,10)
      \put(-55,5){
             \qbezier[2](-43,-2)(-40,-2)(-37,-2)
      \put(-32,-2){\vector(1,0){24}}
      \put(1,0){\HVCenter{\small $\textsf{Q}_3$}}
      \put(9,-2){\vector(1,0){24}}
      \put(20,6){\HVCenter{\footnotesize $f_3$}}
      \put(42,0){\HVCenter{\small $\textsf{Q}_2$}}
      \put(49,-2){\vector(1,0){24}}
      \put(60,6){\HVCenter{\footnotesize $f_{2}$}}
      \put(82,0){\HVCenter{\small $\textsf{Q}_1$}}
      \put(90,-2){\vector(1,0){24}}
      \put(102,6){\HVCenter{\footnotesize $f_1$}}
      \put(140,0){\HVCenter{\small $\textsf{Q}$\footnotesize $(S_i)= P_i$}}
      \put(166,-2){\vector(1,0){24}}
      \put(176,6){\HVCenter{\footnotesize $f_0$}}
      \put(198,0){\HVCenter{\footnotesize $S_i$}}
      \put(205,-2){\vector(1,0){24}}
      \put(235,-1){\HVCenter{\footnotesize $0$}}
     }
  \end{picture}
\end{equation} \vspace{.01in}

\noindent be the minimal projective resolution of $A$-module
simple $S_i$.

\begin{enumerate}
\item[(1)] The projective $\textsf{Q}_1$ is the direct sum of indecomposable projective
$P_a$, where there is an arrow from the vertex $i$ to the vertex
$a$.

\item[(2)] If $P_j$ is a direct summand of the projective $\textsf{Q}_{r+1}\   (r \geq
2)$, then  $S_j \in \mathrm{Top}$$\: \textsf{Q}_{r+1} =
\mathrm{Top}$$\: \mathrm{Ker}$$\ f_r$ and  $S_j$ is a composition
factor of $\textsf{Q}_r$. Moreover, if  $S_j$ is not a composition
factor of  $\mathrm{Ker}$$\: f_{r-1}$, then for all  $h$ such that
$S_h \in \mathrm{Top}$$\: \textsf{Q}_{r-1}$, it follows that all
paths from $h$ to $j$  are zero paths in $A$.

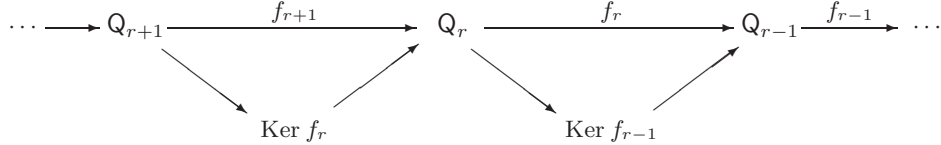
  \begin{picture}(80,65)
      \put(7,10){
      \put(-34,38){\vector(1,0){20}}
      \put(-42,38){\HVCenter{\footnotesize $\cdots$}}
      \put(0,40){\HVCenter{\small $\textsf{Q}_{r+1}$}}
      \put(120,40){\HVCenter{\small $\textsf{Q}_{r}$}}
      \put(60,0){\HVCenter{\footnotesize $\mbox{Ker}\: f_{r}$}}
      \put(12,38){\vector(1,0){94}}
      \put(10,30){\vector(4,-3){32}}
      \put(75,8){\vector(4,3){32}}

      \put(132,38){\vector(1,0){94}}
      \put(240,40){\HVCenter{\small $\textsf{Q}_{r-1}$}}

      \put(252,38){\vector(1,0){36}}
      \put(300,38){\HVCenter{\footnotesize $\cdots$}}

      \put(180,0){\HVCenter{\footnotesize $\mbox{Ker}\: f_{r-1}$}}
      \put(127,30){\vector(4,-3){32}}
      \put(196,7){\vector(4,3){32}}

      \put(180,45){\HVCenter{\footnotesize $f_{r}$}}
      \put(60,45){\HVCenter{\footnotesize $f_{r+1}$}}
      \put(270,45){\HVCenter{\footnotesize $f_{r-1}$}}
      }
  \end{picture}

\item[(3)] Repeating the process, we obtain that, for any indecomposable projective $A$-module $P_d$ which is a
direct summand of any term of  \textsc{(\ref{rp1})}, we have that
there exists a path from $i$ to $d$ in the quiver $Q$.
\end{enumerate}

\end{remark}

From Remark \ref{rem:rp1}, we get a description of the first and
second terms of the minimal projective resolution (\ref{rp1}). We
continue studying the behavior of some of the other terms of this
resolution.

From now on, we denote by  $\mu_{M}(S)$  the multiplicity of the
simple module  $S$ as a composition factor of the $A$-module $M$.

The following proposition describes the term $\textsf{Q}_2$.

\begin{prop}\label{prop:lemaA}
Let $A = kQ / I $ be a strongly simply connected schurian algebra.
Let $i$ and $b$ be vertices of $Q$, and let

  \begin{picture}(80,40)
      \put(110,18){
             \qbezier[3](-90,-2)(-86,-2)(-82,-2)
      \put(-77,-2){\vector(1,0){32}}
      \put(-35,0){\HVCenter{\small $\textsf{Q}_3$}}
      \put(-27,-2){\vector(1,0){32}}
      \put(-10,6){\HVCenter{\footnotesize $f_3$}}
      \put(16,0){\HVCenter{\small $\textsf{Q}_2$}}
      \put(25,-2){\vector(1,0){32}}
      \put(40,6){\HVCenter{\footnotesize $f_{2}$}}
      \put(68,0){\HVCenter{\small $\textsf{Q}_1$}}
      \put(78,-2){\vector(1,0){32}}
      \put(93,6){\HVCenter{\footnotesize $f_1$}}
      \put(122,0){\HVCenter{\small $\textsf{Q}_0$}}
      \put(133,-2){\vector(1,0){32}}
      \put(147,6){\HVCenter{\footnotesize $f_0$}}
      \put(175,0){\HVCenter{\footnotesize $S_i$}}
      \put(186,-2){\vector(1,0){32}}
      \put(225,-2){\HVCenter{\footnotesize $0$}}
     }
  \end{picture}

\noindent be the minimal projective resolution of simple
$A$-module $S_i$ . Then the following statements are equivalent:
\begin{enumerate}
\item $P_b$ is a direct summand of  $\textsf{Q}_2$,
\item there is a minimal relation from $i$ to $b$.
\end{enumerate}

\end{prop}

\begin{proof} If $P_b$ is a direct summand of
$\textsf{Q}_2$, then there exist a vertex  $a_1 \neq b$ such that
$P_{a_1}$ is a direct summand of $\textsf{Q}_1$ and a nonzero path
$w_1: a_1 \leadsto b$ in the algebra $A$.

Let $a_1, a_2, \ldots, a_k$ all vertices that satisfy: $P_{a_j}$
is a direct summand of  $\textsf{Q}_1$, $w_j:a_j \leadsto b$ is a
nonzero path in $A$, for $j=1, \ldots,k$. We denote $\alpha_j$ the
corresponding arrow from the vertex $i$ to vertex $a_j$. Then, in
the quiver  $Q$, we have the following situation:

  
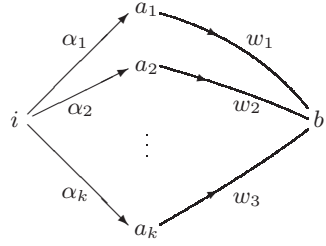
\begin{picture}(80,100)
      \put(130,30){
      \put(-5,20){\HVCenter{\footnotesize $i$}}
      \put(45,62){\HVCenter{\footnotesize $a_1$}}
      \put(45,40){\HVCenter{\footnotesize $a_2$}}
      \put(45,-22){\HVCenter{\footnotesize $a_k$}}
      \put(110,20){\HVCenter{\footnotesize $b$}}
      \put(0,22){\vector(1,1){38}}
      \put(18,50){\HVCenter{\footnotesize $\alpha_1$}}
      \put(2,21){\vector(2,1){36}}
      \put(20,24){\HVCenter{\footnotesize $\alpha_2$}}
      \put(0,18){\vector(1,-1){38}}
      \put(18,-8){\HVCenter{\footnotesize $\alpha_k$}}
             \qbezier[2](45,14)(45,10)(45,6)

      \qbezier(50,60)(85,50)(106,24)
      \put(88,50){\HVCenter{\footnotesize $w_1$}}
      \qbezier(50,40)(85,30)(106,20)
      \put(83,25){\HVCenter{\footnotesize $w_2$}}
      \qbezier(50,-20)(85,0)(106,16)
      \put(83,-10){\HVCenter{\footnotesize $w_3$}}

      \put(72,51.3){\vector(2,-1){0.1}}
      \put(68,34.3){\vector(3,-1){0.1}}
      \put(72,-7.4){\vector(2,1){0.1}}
      }
  \end{picture}

If $\alpha_j w_j = 0$, for all  $j=1, \ldots,k$, then there are
monomial relations $\rho_1, \rho_2, \ldots, \rho_k$ in $I$, such
that $s(\rho_j)=i$ and $t(\rho_j)=q_j$, for all $j=1, \ldots,k$,
where  $a_j > q_j \geq b$ and $q_j$ is a vertex of the path $w_j$.
We can assume that $\rho_1, \rho_2, \ldots, \rho_k$ are minimal
relations in $I$.

If some $q_j = b$, we obtain the result. Suppose that  $q_j \neq
b$, for all $j=1, \ldots,k$. Then, for each $j=1, \ldots,k$, it
follows that  $S_{q_j}$ is not a composition factor of
$\mbox{rad}\ P_i$ and $S_{q_j} \in \mbox{Top Ker}\  f_1 =
\mbox{Top}\ \textsf{Q}_2$, which is a contradiction. Therefore,
there exists a minimal relation that starts at $i$ and ends at
$b$.

Now, if for some  $1 \leq j\leq k$, $\alpha_j w_j \neq 0$, since
$A$ is a quotient of an incidence algebra, it follows that all
parallel paths
 to $\alpha_j w_j$ are nonzero. In particular, we get that
$\alpha_{j+1} w_{j+1} \neq 0$ and also $\alpha_j w_j -
\alpha_{j+1} w_{j+1} \in I$. Then there is a commutativity
relation from  $i$ to $b$. A similar argument shows that some of
these $k$ commutativity relations must be minimal.

Conversely, let  $\rho:i \leadsto b $ be a minimal relation. Then
we have again a similar situation to the one described in the
previous figure, where the paths  $w_1, w_2, \ldots, w_k$  could
be zero paths in the algebra  $A$.

If $\rho$ is a monomial relation, then we can assume that $\rho =
\alpha_1 w_1 =0$ in $A$ and $w_1 \neq 0$ in $A$. Since $A$ is a
quotient of an incidence algebra, it follows that  $\alpha_j w_j
=0$ in $A$, for $1 \leq j \leq k$, and we get that  $S_b \notin
\mbox{rad}\  P_i$. Moreover, since  $\rho$ is a minimal relation,
it must be $S_b \in \mbox{Top Ker}\  f_1 = \mathrm{Top}\
\textsf{Q}_2$. Consequently, $P_b$ is a direct summand of
$\textsf{Q}_2$.

Now suppose that  $\rho$ is a minimal commutativity relation.
Without loss of generality, we can assume that $\rho = \alpha_1
w_1 - \alpha_2 w_2$, with $\alpha_1 w_1 \neq 0$, $\alpha_2 w_2
\neq 0$.

Then $S_b$ is a composition factor of  $P_i$. Then, since $A$ is
schurian, $\mu_{{\scriptsize \mbox{rad}} P_i}(S_b)$ is either $0$
or $1$. Since $\mu_{P_i}(S_b) = k$, then $\mu_{{\scriptsize
\mbox{Ker}} f_1}(S_b) \geq 1$. Hence the minimality of $\rho$
implies that  $S_b \in \mbox{Top \  Ker}\ f_1$; i.e., $P_b$ is a
direct summand of $\textsf{Q}_2$.
\end{proof}

Next, we recall the notion of convex hull of two vertices and some
useful results.

Let  $A = kQ/I$ be an algebra  and let $i,j$ be vertices of  $Q$.
The  \textit{convex hull between $i$ and $j$}, $\mbox{Conv}(i,j) =
kQ'/I'$ is the subalgebra of $A$ given by the quiver $Q'$
\begin{itemize}
\item $(Q')_0 = \{ k \in Q_0 \; / \mbox{ there are walks } i \leadsto k \leadsto j \}$
\item $(Q')_1 = \{ \alpha \in Q_1 \; / s(\alpha) \mbox{ and } t(\alpha) \in (Q')_0 \}$
\end{itemize}

\noindent and $I'$ is generated by induced relations.

Note that $C = \mbox{Conv}(i,j)$, as $k$-category is a full and
convex subcategory of $A$. Under these conditions, it follows from
\cite{apt} that $\mbox{Ext}^i_A (X,Y) \cong \mbox{Ext}^i_C (X,Y)$
for all  $i \geq 0$ and $X,Y \in \mbox{mod} \: C$.

It follows that, if the algebra $A$ is  strongly simply connected
and schurian, then so is $C = \mbox{Conv}(i,j)$.

\begin{lem}
Let $A = kQ / I $ be a strongly simply connected schurian algebra,
let $i$ and  $j$ be  vertices of $Q$, and let

  \begin{picture}(80,40)
      \put(120,18){
             \qbezier[3](-90,-2)(-86,-2)(-82,-2)
      \put(-77,-2){\vector(1,0){32}}
      \put(-35,0){\HVCenter{\small $\textsf{Q}_3$}}
      \put(-27,-2){\vector(1,0){32}}
      \put(-10,6){\HVCenter{\footnotesize $f_3$}}
      \put(16,0){\HVCenter{\small $\textsf{Q}_2$}}
      \put(25,-2){\vector(1,0){32}}
      \put(40,6){\HVCenter{\footnotesize $f_{2}$}}
      \put(68,0){\HVCenter{\small $\textsf{Q}_1$}}
      \put(78,-2){\vector(1,0){32}}
      \put(93,6){\HVCenter{\footnotesize $f_1$}}
      \put(122,0){\HVCenter{\small $\textsf{Q}_0$}}
      \put(133,-2){\vector(1,0){32}}
      \put(147,6){\HVCenter{\footnotesize $f_0$}}
      \put(175,0){\HVCenter{\footnotesize $S_i$}}
      \put(186,-2){\vector(1,0){32}}
      \put(225,-2){\HVCenter{\footnotesize $0$}}
     }
  \end{picture}

\noindent be the minimal projective resolution of the simple
$A$-module $S_i$. If $P_j$ is a direct summand of $\textsf{Q}_3$,
then $S_j$ is not a composition factor of the term $\textsf{Q}_0$.
\end{lem}

\begin{proof} Suppose $S_j$ is a composition factor of the term $\textsf{Q}_0$,
where  $\textsf{Q}_0$ is the  indecomposable projective module
associated with the vertex  $i$. Since the algebra $A$ is
schurian, we have that $\mu_{\textsf{Q}_0}(S_j) = 1$. Then in $C =
\mathrm{Conv}$$(i,j)$ there is no monomial relations. Therefore,
$C$ is an incidence algebra. Moreover, since  $S_j \in
\mathrm{Top}$$\: \textsf{Q}_3$, there exists a morphism  $\pi:
\textsf{Q}_3 \longrightarrow S_j$. Considering the push out of
$\pi$ and $f_3$ we have the following commutative diagram:

  \begin{picture}(80,60)
      \put(112,47){
             \qbezier[3](-90,-2)(-86,-2)(-82,-2)
      \put(-77,-2){\vector(1,0){32}}
      \put(-35,0){\HVCenter{\small $\textsf{Q}_3$}}
      \put(-27,-2){\vector(1,0){32}}
      \put(-10,6){\HVCenter{\footnotesize $f_3$}}
      \put(16,0){\HVCenter{\small $\textsf{Q}_2$}}
      \put(25,-2){\vector(1,0){32}}
      \put(40,6){\HVCenter{\footnotesize $f_{2}$}}
      \put(68,0){\HVCenter{\small $\textsf{Q}_1$}}
      \put(78,-2){\vector(1,0){32}}
      \put(93,6){\HVCenter{\footnotesize $f_1$}}
      \put(122,0){\HVCenter{\small $\textsf{Q}_0$}}
      \put(133,-2){\vector(1,0){32}}
      \put(147,6){\HVCenter{\footnotesize $f_0$}}
      \put(175,0){\HVCenter{\footnotesize $S_i$}}
      \put(186,-2){\vector(1,0){32}}
      \put(225,-2){\HVCenter{\footnotesize $0$}}

      \put(-83,-42){\HVCenter{\footnotesize $0$}}
            \put(-77,-42){\vector(1,0){32}}
      \put(-35,-40){\HVCenter{\footnotesize $S_j$}}
      \put(-27,-42){\vector(1,0){32}}
      \put(16,-40){\HVCenter{\small $\textsf{Q'}$}}
      \put(25,-42){\vector(1,0){32}}
      \put(68,-40){\HVCenter{\small $\textsf{Q}_1$}}
      \put(78,-42){\vector(1,0){32}}
      \put(122,-40){\HVCenter{\small $\textsf{Q}_0$}}
      \put(133,-42){\vector(1,0){32}}
      \put(175,-40){\HVCenter{\footnotesize $S_i$}}
      \put(186,-42){\vector(1,0){32}}
      \put(225,-42){\HVCenter{\footnotesize $0$}}

      \put(175,-7){\vector(0,-1){26}}
      \put(183,-20){\HVCenter{\footnotesize $Id$}}

      \put(122,-7){\vector(0,-1){26}}
      \put(130,-20){\HVCenter{\footnotesize $Id$}}

      \put(68,-7){\vector(0,-1){26}}
      \put(76,-20){\HVCenter{\footnotesize $Id$}}

      \put(16,-7){\vector(0,-1){26}}

      \put(-35,-7){\vector(0,-1){26}}
      \put(-27,-20){\HVCenter{$\pi$}}
     }
  \end{picture}

\noindent with exact rows.

Since $S_j \in \mathrm{Top\  Ker}$$\: f_2$, it follows that  $\pi$
does not factor through  $\textsf{Q}_2$; and consequently,
$\mathrm{Ext}$$^3_A (S_i, S_j) \neq 0$. Since $\mathrm{Ext}$$^3_C
(S_i, S_j) \cong \mathrm{Ext}$$^3_A (S_i, S_j)$, we have that
$dp_C\: S_i = 3$.

Then,  $C$ is an incidence algebra with global dimension at least
three. By \cite{iz}, either $C = k \Sigma$ must contain some graph
$Q_n$, for $n \geq 3$, or there is a subposet $\Sigma'$ of
$\Sigma$ isomorphic  to $Q_2$ that is not contained in any
subposet of the form

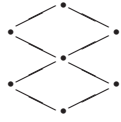
\begin{picture}(50,45)
    \put(150,-10){
      \put(0,40){\circle*{2}}
      \put(20,50){\circle*{2}}
      \put(40,40){\circle*{2}}
      \put(20,30){\circle*{2}}
      \put(0,20){\circle*{2}}
      \put(40,20){\circle*{2}}
      \put(20,10){\circle*{2}}
      \put(2,42){\line(2,1){15}}
      \put(22,31){\line(2,1){15}}
      \put(2,21){\line(2,1){15}}
      \put(22,11){\line(2,1){15}}
      \put(2,38){\line(2,-1){15}}

      \put(22,49){\line(2,-1){15}}
      \put(22,28){\line(2,-1){15}}
      \put(2,18){\line(2,-1){15}}
    }
\end{picture}

\noindent In both cases  $C$ is not strongly simply connected.
Therefore, there is a monomial  relation in $C$, and we get that
$S_j$ is not a composition factor of  $\textsf{Q}_0$.
\end{proof}

From now on, we consider the convex hull $C = \mbox{Conv}(i,j)
\cong kQ_C / I_C$ between two vertices $i$ and $j$ such that:
\begin{itemize}
\item $pd_C \: S_k < pd_C \: S_i$, for all $k \in (Q_C)_0$.
\item $\mbox{Ext}^3_A (S_i, S_j) \neq 0$, and if  $\mbox{Ext}^3_A (S_i, S_k) \neq
0$ for $k \in (Q_C)_0, \; k \neq j$, then $k \not> j$.
\end{itemize}

\begin{lem}\label{lem:pozo-fuente}
Let $A= kQ/I$ be a strongly simply connected schurian algebra. Let
$S_i$ be a simple $A$-module with  $pd_A\: S_i =3$ and let

  \begin{picture}(80,40)
      \put(110,18){
      \put(-82,-1){\HVCenter{\footnotesize $0$}}
      \put(-77,-2){\vector(1,0){32}}
      \put(-35,0){\HVCenter{\small $\textsf{Q}_3$}}
      \put(-27,-2){\vector(1,0){32}}
      \put(-10,6){\HVCenter{\footnotesize $f_3$}}
      \put(16,0){\HVCenter{\small $\textsf{Q}_2$}}
      \put(25,-2){\vector(1,0){32}}
      \put(40,6){\HVCenter{\footnotesize $f_{2}$}}
      \put(68,0){\HVCenter{\small $\textsf{Q}_1$}}
      \put(78,-2){\vector(1,0){32}}
      \put(93,6){\HVCenter{\footnotesize $f_1$}}
      \put(122,0){\HVCenter{\small $\textsf{Q}_i$}}
      \put(133,-2){\vector(1,0){32}}
      \put(147,6){\HVCenter{\footnotesize $f_0$}}
      \put(175,0){\HVCenter{\footnotesize $S_i$}}
      \put(186,-2){\vector(1,0){32}}
      \put(225,-2){\HVCenter{\footnotesize $0$}}
     }
  \end{picture}

\noindent be the minimal projective resolution of $S_i$ in $A$.
Let $P_j$ be a direct summand $\textsf{Q}_3$ and $C =
\mathrm{Conv}$$(i,j)$. Then, $dp_C \: S_i = 3$ in $C$.

\end{lem}

\begin{proof} From the previous lemma, it follows that
$\mbox{Ext}^3_A (S_i,S_j) \neq 0$ and $\mbox{Ext}^4_A (S_i,S_j) =
0$. Since $C$ is a full and convex subcategory of $A$, we have
that
 $\mbox{Ext}^3_C (S_i,S_j) \neq 0$ and $\mbox{Ext}^4_C
(S_i,S_j) = 0$. Therefore,  $dp_C \: S_i = 3$. \end{proof}

\begin{prop}\label{prop:lemaB}
Let $A = kQ / I $ be a strongly simply connected schurian algebra.
Let $i$ and $j$ be vertices of  $Q$, and let

  \begin{picture}(80,40)
      \put(110,18){
             \qbezier[3](-90,-2)(-86,-2)(-82,-2)
      \put(-77,-2){\vector(1,0){32}}
      \put(-35,0){\HVCenter{\small $\textsf{Q}_3$}}
      \put(-27,-2){\vector(1,0){32}}
      \put(-10,6){\HVCenter{\footnotesize $f_3$}}
      \put(16,0){\HVCenter{\small $\textsf{Q}_2$}}
      \put(25,-2){\vector(1,0){32}}
      \put(40,6){\HVCenter{\footnotesize $f_{2}$}}
      \put(68,0){\HVCenter{\small $\textsf{Q}_1$}}
      \put(78,-2){\vector(1,0){32}}
      \put(93,6){\HVCenter{\footnotesize $f_1$}}
      \put(122,0){\HVCenter{\small $\textsf{Q}_0$}}
      \put(133,-2){\vector(1,0){32}}
      \put(147,6){\HVCenter{\footnotesize $f_0$}}
      \put(175,0){\HVCenter{\footnotesize $S_i$}}
      \put(186,-2){\vector(1,0){32}}
      \put(225,-2){\HVCenter{\footnotesize $0$}}
     }
  \end{picture}

\noindent be the minimal projective resolution of $S_i$ in $A$.
Consider the following set of vertices:

\begin{itemize}

\item $\mathcal{R }$$= \{b \in Q_0 : P_b $ is a direct summand of
$\textsf{Q}_2$ and there is a nonzero path $b \leadsto j \}$,

 \item $\mathcal{S}$ $= \{a \in Q_0 : P_a $ is a direct summand of $\textsf{Q}_1$ and there is $b \in \mathcal{R}$ such that  $a
\leadsto b$ is a nonzero path$\}$,

\item $r = \mathrm{Card}\:$$\mathcal{R }$ and $s =
\mathrm{Card}\:$$\mathcal{S }$.
\end{itemize}

If $s \geq r$, then the following statements are equivalent:

\begin{enumerate}
\item $P_j$ is a direct summand of $\textsf{Q}_3$,
\item there are at least  $s-r+1$  monomial relations $a \leadsto j$ with $a \in \mathcal{S}$ and
at least one (monomial or commutative) relation  from $a \in
\mathrm{S}$ to the vertex $j$ is a minimal relation.
\end{enumerate}
\end{prop}

\begin{proof} Suppose that  $P_j$ is a direct summand of  $P_3$, $\mathcal{S}$$= \{a_1,$ $ \ldots , a_s\}$
and $\mathcal{R}$$= \{b_1, \ldots , b_r\}$. If $\mu_{P_{a_k}}
(S_j) = 1$, for all $1 \leq k \leq s$, then
$$\mu_{{\scriptsize \mbox{Ker}} f_1}(S_j) = \mu_{\textsf{Q}_1}(S_j) =
\Sigma^s_{k=1}\: \mu_{P_{a_k}}(S_j) = s \geq r = \Sigma^r_{h=1}\:
\mu_{P_{b_h}}(S_j) = \mu_{\textsf{Q}_2}(S_j),$$ since
$\mu_{P_i}(S_j)= 0$. Since $\mu_{\textsf{Q}_2}(S_j) =
\mu_{{\scriptsize \mbox{Ker}} f_1}(S_j) + \mu_{{\scriptsize
\mbox{Ker}} f_2}(S_j)$, it follows that $\mu_{{\scriptsize
\mbox{Ker}} f_1}(S_j)$ $= s = r = \mu_{\textsf{Q}_2}(S_j)$. This
implies that $\mu_{{\scriptsize \mbox{Ker}} f_2}(S_j) = 0$, a
contradiction. Therefore, there is $a \in \mathcal{S}$ such that
$\mu_{P_{a}} (S_j) = 0$.

Let $\mu_{\textsf{Q}_1} (S_j) = \mu_{{\scriptsize \mbox{Ker}}
f_1}(S_j) = q$ , i.e., there are $s-q$ monomial relations  $a
\leadsto j$ with $a \in \mathcal{S}$. Let $\mu_{\textsf{Q}_3}
(S_j) = \mu_{{\scriptsize \mbox{Ker}} f_2}(S_j)  = \alpha > 0$.
Then
$$r = \mu_{\textsf{Q}_2}(S_j) = \mu_{{\scriptsize \mbox{Ker}} f_1}(S_j) +
\mu_{{\scriptsize \mbox{Ker}} f_2}(S_j) = q + \alpha,$$ implies
that $q < r$. Then, $s-q > s - r \geq 0$. As a consequence, we get
that there are at least  $s-r+1$  monomial relations $a \leadsto
j$ with $a \in \mathcal{S}$. Also, if none of the relations from
$a \in \mathrm{S}$ to $j$ is minimal, then $S_j$ can not be in
$\mbox{Top Ker}\: f_2$.

Conversely, if there are $q \geq s-r+1$ monomial relations $a
\leadsto c$ with $a \in \mathcal{S}$, then
$$\mu_{{\scriptsize \mbox{Ker}} f_1}(S_j) = \mu_{\textsf{Q}_1}(S_j) = s - q
 \geq r - 1.$$ Since $\mu_{\textsf{Q}_2}(S_j) = r$, it follows that $$ r =
\mu_{\textsf{Q}_2}(S_j) = \mu_{{\scriptsize \mbox{Ker}} f_1}(S_j)
+ \mu_{{\scriptsize \mbox{Ker}} f_2}(S_j) \leq r - 1 +
\mu_{{\scriptsize \mbox{Ker}} f_2}(S_j).$$  We get that, $1 \leq
\mu_{{\scriptsize \mbox{Ker}} f_2}(S_j) \leq
\mu_{\textsf{Q}_3}(S_j)$.

Since at least one of the (or monomial or commutative) relation
starting at the vertex $a \in \mathrm{S}$  and ending in the
vertex $j$ is a minimal relation, we have that $S_j \in \mbox{Top
Ker}\: f_2$. Indeed, if the minimal relation $\rho: a \leadsto b
\leadsto h \longrightarrow c $ is monomial, then $\mu_{P_a}(S_h) =
1$, $\mu_{P_a}(S_j) = 0$, $\mu_{P_b}(S_h) = 1$, $\mu_{P_b}(S_j) =
1$. Now, if the minimal relation is a commutative relation

  
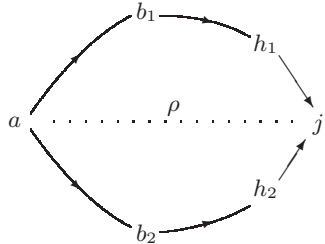
\begin{picture}(80,90)
      \put(130,25){
      \put(-5,20){\HVCenter{\footnotesize $a$}}
      \put(45,62){\HVCenter{\footnotesize $b_1$}}
      \put(45,-22){\HVCenter{\footnotesize $b_2$}}
      \put(90,50){\HVCenter{\footnotesize $h_1$}}
      \put(90,-5){\HVCenter{\footnotesize $h_2$}}
      \put(110,20){\HVCenter{\footnotesize $j$}}
            \qbezier(1,23)(20,50)(39,60)
      \put(20,45.7){\vector(1,1){0.1}}
            \qbezier(1,17)(20,-10)(39,-20)
      \put(20,-5.7){\vector(1,-1){0.1}}
            \qbezier(50,60)(70,60)(84,52)
      \put(70,57.5){\vector(3,-1){0.1}}
             \qbezier(50,-22)(70,-20)(84,-12)
      \put(70,-17.8){\vector(3,1){0.1}}

      \put(95,45){\vector(2,-3){13}}
      \put(95,-2){\vector(2,3){10}}

      \put(55,26){\HVCenter{\footnotesize $\rho$}}
             \qbezier[15](10,20)(55,20)(100,20)
      }
  \end{picture}

\noindent  we have that  $\mu_{P_a}(S_{h_1}) = 1 =
\mu_{P_a}(S_{h_2}) $, $\mu_{P_a}(S_j) = 1$,
$\mu_{P_{b_1}}(S_{h_1}) = 1 = \mu_{P_{b_2}}(S_{h_2}) = 1 $,
$\mu_{P_{b_1}}(S_j) = 1 = \mu_{P_{b_2}}$. In both cases, it is
clear that  $S_j$ must belong to  $\mbox{Top Ker}\: f_2$. Then
$P_j$ is a direct summand of $\textsf{Q}_3$. \end{proof}

Note that, without loss of generality, we can take $s \geq r$.
Otherwise, the opposite algebra satisfies the desired property.

Our purpose now is to study the projective dimensions of the
simple modules. To do this, we need the following technical
lemmas.

\begin{lem}\label{lem:afirm 1}
Let $A = kQ / I $  be a strongly simply connected schurian
algebra. Let

  \begin{picture}(80,40)
      \put(120,18){
      \put(-82,-1){\HVCenter{\footnotesize $0$}}
      \put(-77,-2){\vector(1,0){32}}
      \put(-35,0){\HVCenter{\small $\textsf{Q}_{n}$}}
      \put(-27,-2){\vector(1,0){32}}
      \put(-13,6){\HVCenter{\footnotesize $f_{n}$}}
                   \qbezier[3](10,-2)(15,-2)(20,-2)
      \put(25,-2){\vector(1,0){32}}
      \put(40,6){\HVCenter{\footnotesize $f_{2}$}}
      \put(68,0){\HVCenter{\small $\textsf{Q}_1$}}
      \put(78,-2){\vector(1,0){32}}
      \put(93,6){\HVCenter{\footnotesize $f_1$}}
      \put(122,0){\HVCenter{\small $\textsf{Q}_0$}}
      \put(133,-2){\vector(1,0){32}}
      \put(147,6){\HVCenter{\footnotesize $f_0$}}
      \put(175,0){\HVCenter{\footnotesize $S_i$}}
      \put(186,-2){\vector(1,0){32}}
      \put(225,-2){\HVCenter{\footnotesize $0$}}
     }
  \end{picture}

\noindent be the minimal projective resolution of simple
$A$-module $S_i$. Then  $S_i$ is not a composition factor of the
$A$-modules $\mathrm{rad}$$\: \textsf{Q}_0, \textsf{Q}_1, \ldots ,
\textsf{Q}_n$.
\end{lem}

\begin{proof} If $S_i \in \mbox{rad}\ \textsf{Q}_0 =
\mbox{Ker}\  f_0$, then there exists a path of length at least one
from $i$ to $i$ in $Q$, which is a contradiction since $A$ is
triangular.

Suppose now that  $S_i$ appears as a composition factor in some
term $\textsf{Q}_j$, with $j \geq 1$. Let $0 \neq k = \mbox{min}
\{ l \in \{1, \ldots,n\} \mbox{ such that } \mu_{P_l}(S_i) \neq 0
\}$. Then $\mu_{{\scriptsize \mbox{Ker}} f_{k-1} }(S_i) = 0$ and
there is $S_k$ in $\mbox{Top } \textsf{Q}_k = \mbox{Top Ker }
f_{k-1}$, such that:

\begin{itemize}
\item[(k)] there is a nonzero path from $k$ to $i$
\item[(k-1)] there is  $S_{k-1} \in \mbox{Top } \textsf{Q}_{k-1} = \mbox{Top
Ker } f_{k-2}$ such that there is a nonzero path from  $(k-1)$ to
$k$
\end{itemize}

Consider $0 \leq h \leq k-1$:

\begin{itemize}
\item[(k-h)] there is $S_{k-h} \in \mbox{Top } \textsf{Q}_{k-h} = \mbox{Top
Ker } f_{k-h-1}$ such that there is a non-zero path $w_h$ from
${k-h}$ to ${k-h+1}$
\end{itemize}

We obtain  $S_1 \in \mbox{Top Ker } f_0 = \mbox{Top rad }
\textsf{Q}_0$, and therefore there is an arrow from $i$ to $1$.
Consequently,  $$ i \longrightarrow 1 \leadsto \cdots \leadsto
(k-h)\leadsto (k-1) \leadsto k \leadsto i$$ is a cycle in $Q$,
which is a contradiction since $A$ is triangular.
\end{proof}

\begin{lem}\label{lem:afirm 2}
Let $A = kQ / I $ a strongly simply connected schurian algebra
with $\mathrm{gl.dim.}$$\: A = n$. Consider $J = \{h \in Q_0 :
pd_A S_h = n \}$. Then there exists $j \in J$ such that $S_h$
 do not appear in the  minimal projective resolution of
$S_j$, with $h \in J \backslash \{j\}$.
\end{lem}

\begin{proof} Suppose that, for all $j \in J$,
there is   $h \in J$ such that $S_h$ appears as a composition
factor in the minimal projective resolution of $S_j$. Let
$\textsf{Q}_k$ be the first term of the minimal projective
resolution of  $S_j$ in which $S_h$  appears as composition
factor, then $\mu_{{\scriptsize \mbox{Ker}} f_{k-1} }(S_h) = 0$.
An analogous argument to the one used in Lemma \ref{lem:afirm 1},
shows that there is a path  $j \leadsto h$ and also $h \neq j$.

Next, we will renumber the elements of  $J$. Let  $j \in J$ and
$j_1 \in J \backslash \{j\}$ such that $S_{j_1}$ appears in the
minimal projective resolution of  $S_j$. Now we consider $j_2 \in
J \backslash \{j_1\}$ such that $S_{j_2}$ appears in the minimal
projective resolution of  $S_{j_1}$. If $j_2 = j$ we have a path
in $Q$. It follows that  $j_2 \neq j$ and there is $j_3 \in J
\backslash \{j, j_1, j_2\}$ such that $S_{j_3}$ appears in the
minimal projective resolution of $S_{j_2}$. Therefore,
$$J = \{ j, j_1, \ldots, j_{m-1} \}$$ and there exists a sequence of paths in  $Q$:
$$ j \leadsto j_1 \leadsto j_2 \leadsto \cdots \leadsto j_{m-1}$$

Since $S_h$ appears in  the minimal projective resolution of
$S_{j_{m-1}}$ , for some $h \in J$, we get that  either $h=j$, or
$h=j_l$, with $1 \leq l\leq m-1$, are a contradiction with $A$
triangular.

Consequently, there is a $j \in J$ such that any $S_h,  h \in J
\backslash \{j\}$, does not appear as composition factor in the
minimal projective resolution of  $S_j$. \end{proof}

\begin{prop}\label{prop:dpM 2}
Let $A$ be an algebra with $\mathrm{gl.dim.}$$\:  A = n$ and let
$M$ be an $A$-module. If $pd\: M = n $, then there is a
composition factor $S$ of $M$, such that $pd\: S = n$.
\end{prop}

\begin{proof} The proof follows by induction on the Loewy length $l_w(M)$ of $M$.

If $l_w(M) = 1$ then $M$ is semisimple, and it follows that one of
its direct summands  has projective dimension $n$. Suppose now
that the result holds for  $l_w(M) < m$. Consider the exact
sequence
  \begin{center}
  $\begin{array}{ccccccccc}
  0 & \rightarrow & \mbox{rad}\  M & \rightarrow & M & \rightarrow & \mbox{Top}\  M & \rightarrow & 0 \
  \end{array}$
  \end{center}
 then we have that $l_w(\mathrm{rad}\  M) = m-1$ and $n = pd\:  M \leq
\mbox{sup}\{pd\:  \mbox{rad}\:  M, pd\:  \mbox{Top}\  M \}$.

Since $\mbox{gl.dim.}\:  A = n$, then $\mbox{sup}\{pd\;
\mbox{rad}\; M, \; pd\;  \mbox{Top}\;  M \} \leq n$. Therefore $n
=~pd\: M = \mbox{sup} \{ pd\:  \mbox{rad}\:  M, pd\: \mbox{Top}\:
M \}$.

If $pd\:  \mbox{Top}\:  M = n$, then  since the module
$\mbox{Top}\: M$ is semisimple, the result follows. Otherwise,
$pd\: \mbox{rad}\: M = n$,  and by inductive hypothesis there is
$S$ composition factor $\mbox{rad}\:  M$ such that $pd\: S = n$.
Since $S$ is also a composition factor of $M$, the result follows.
\end{proof}
Now, we are able to state the following result.

\begin{thm}\label{thm:reduccion de la dim gl}
Let $A = kQ / I $ be a  strongly simply connected schurian algebra
with $\mathrm{gl.dim. }$$\:  A = n$.  Then, for all $0 \leq m \leq
n$, there is a simple $A$-module $S$ such that $pd_A\ S = m$.
\end{thm}

\begin{proof} We consider again the set
$J = \{h \in Q_0 : pd_A S_h = n \}$. As  $n = \mbox{gl.dim.}\: A =
\mbox{sup}\{ pd_A S_i : i \in Q_0\}$, it follows that $J \neq
\emptyset$.

Since $A$ is triangular, it follows that the algebra $B = A /
\langle  \{ e_k : k \in J \} \rangle$ is also triangular, and we
have the functorial immersion  $\mbox{mod }B \hookrightarrow
\mbox{mod }A$.

By Lemma  \ref{lem:afirm 2}, there is $j \in J$ such that if

  \begin{picture}(80,42)
      \put(110,35){
      \put(-82,-1){\HVCenter{\footnotesize $0$}}
      \put(-77,-2){\vector(1,0){32}}
      \put(-35,0){\HVCenter{\small $\textsf{Q}_{n}$}}
      \put(-27,-2){\vector(1,0){32}}
                   \qbezier[3](10,-2)(15,-2)(20,-2)
      \put(25,-2){\vector(1,0){32}}
      \put(68,0){\HVCenter{\small $\textsf{Q}_1$}}
      \put(78,-2){\vector(1,0){32}}
      \put(122,0){\HVCenter{\small $\textsf{Q}_0$}}
      \put(133,-2){\vector(1,0){32}}
      \put(175,0){\HVCenter{\footnotesize $S_j$}}
      \put(186,-2){\vector(1,0){32}}
      \put(225,-2){\HVCenter{\footnotesize $0$}}

      \put(95,-30){\HVCenter{\footnotesize $\mbox{rad}\:$\small$\textsf{Q}_0$}}
      \put(70,-9){\vector(1,-2){9}}
      \put(105,-23){\vector(1,2){7}}
     }
  \end{picture}

\noindent is the minimal projective resolution of $S_j$, then
$\mu_{\bigoplus^n_{t=0} \textsf{Q}_t} (S_k) = 0$, for all $k \in J
\backslash \{j\}$. Hence $ 0 \longrightarrow  \textsf{Q}_n
\longrightarrow \cdots \longrightarrow \textsf{Q}_1
\longrightarrow \mbox{rad } \textsf{Q}_0 \longrightarrow 0$ is the
projective resolution of the module $\mbox{rad } \textsf{Q}_0$ in
$A$ and also in $B$, because any  $S_k$, $k \in J $, does not
appear as a composition factor of the projective modules
$\textsf{Q}_i, 1\leq i \leq n$. Therefore, $ pd_A \mbox{ rad }
\textsf{Q}_0 = pd_B \mbox{ rad } \textsf{Q}_0 = n-1$.

Let $h \notin J$, the simple $B$-module $S_h$ is also a simple
$A$-module  with  $ s = pd_A \: S_h \leq n-1$.

Consider the projective resolutions of $S_h$ in $\mbox{mod} \: A$
and $\mbox{mod} \: B$, respectively.

  \begin{picture}(80,87)
      \put(130,37){
      \put(-113,39){\HVCenter{\footnotesize $0$}}
      \put(-102,38){\vector(1,0){24}}
      \put(-69,40){\HVCenter{\small $\textsf{Q}^A_s$}}
      \put(-61,38){\vector(1,0){24}}
      \put(-22,40){\HVCenter{\small $\textsf{Q}^A_{s-1}$}}
      \put(-10,38){\vector(1,0){24}}
                   \qbezier[2](21,38)(25,38)(29,38)
      \put(34,38){\vector(1,0){24}}
      \put(68,40){\HVCenter{\small $\textsf{Q}^A_1$}}
      \put(78,38){\vector(1,0){32}}
      \put(122,40){\HVCenter{\small $\textsf{Q}^A_0$}}
      \put(133,38){\vector(1,0){32}}
      \put(150,46){\HVCenter{\footnotesize $f^A_0$}}
      \put(175,40){\HVCenter{\footnotesize $S_h$}}
      \put(185,38){\vector(1,0){24}}
      \put(215,39){\HVCenter{\footnotesize $0$}}

      \put(175,32){\vector(0,-1){25}}
      \put(184,20){\HVCenter{\footnotesize $Id$}}

      \put(-113,-1){\HVCenter{\small $\textsf{Q}^B_{s+1}$}}
      \put(-102,-2){\vector(1,0){24}}
      \put(-69,0){\HVCenter{\small $\textsf{Q}^B_s$}}
      \put(-61,-2){\vector(1,0){24}}
      \put(-22,0){\HVCenter{\small $\textsf{Q}^B_{s-1}$}}
      \put(-10,-2){\vector(1,0){24}}
                   \qbezier[2](21,-2)(25,-2)(29,-2)
      \put(34,-2){\vector(1,0){24}}
      \put(68,0){\HVCenter{\small $\textsf{Q}^B_1$}}
      \put(78,-2){\vector(1,0){32}}
      \put(122,0){\HVCenter{\small $\textsf{Q}^B_0$}}
      \put(133,-2){\vector(1,0){32}}
      \put(150,6){\HVCenter{\footnotesize $f^B_0$}}
      \put(175,0){\HVCenter{\footnotesize $S_h$}}
      \put(185,-2){\vector(1,0){24}}
      \put(215,-2){\HVCenter{\footnotesize $0$}}

      \put(95,-30){\HVCenter{\footnotesize $\mbox{rad}\:$\small $\textsf{Q} P^B_0$}}
      \put(70,-9){\vector(1,-2){9}}
      \put(105,-23){\vector(1,2){7}}

     }
  \end{picture}

Since $\textsf{Q}^B_0$ and $S_h$ are $B$-module of finite type,
$f^B_0$ is a epimorphism and $\mbox{Ker}\: f^B_0 = \mbox{rad} \:
\textsf{Q}^B_0$, then $f^B_0$ is an essential epimorphism in $B$.
Because  $\mbox{rad}_A \: \textsf{Q}^B_0 = \mbox{rad}_B \:
\textsf{Q}^B_0$, and the functorial immersion of the corresponding
category modules, it follows that $f^B_0$ is an essential
epimorphism in $A$. In analogous way, if we consider the
epimorphism $\varphi^B_i : \textsf{Q}^B_i \longrightarrow
\mbox{Ker}\: f^B_{i-1}$ in $\mbox{mod} \: B$, as  $\mbox{Ker}\:
\varphi^B_i \subset \mbox{rad}\: \textsf{Q}^B_i$, then
$\varphi^B_i$ is essential in  $B$ and therefore also on $A$.

Moreover, as $\textsf{Q}^A_0$ is a projective $A$-module, then
there is a morphism of $A$-modules $h_0: \textsf{Q}^A_0
\longrightarrow \textsf{Q}^B_0$ such that $f^B_0 h_0 = f^A _0$,
i.e., the following diagram is commutative.

  \begin{picture}(80,70)
      \put(15,10){

      \put(175,60){\HVCenter{\small $\textsf{Q}^A_0$}}

      \put(175,50){\vector(0,-1){40}}
      \put(183,30){\HVCenter{\footnotesize $f^A_0$}}

      \put(122,0){\HVCenter{\small $\textsf{Q}^B_0$}}
      \put(133,-2){\vector(1,0){32}}
      \put(150,6){\HVCenter{\footnotesize $f^B_0$}}
      \put(175,0){\HVCenter{\footnotesize $S_h$}}
      \put(186,-2){\vector(1,0){32}}
      \put(225,-2){\HVCenter{\footnotesize $0$}}

                   \qbezier[16](164,52)(147,31)(130,10)
      \put(130,10){\vector(-1,-1){.1}}
      \put(140,35){\HVCenter{\footnotesize $h_0$}}
     }
  \end{picture}

Since $f^B_0$ is essential in $A$ and $f^B_0 h_0 = f^A _0$ is a
epimorphism, it follows that  $h_0$ is a epimorphism. Then there
is a morphism of $A$-module $\widetilde{h_0}: \mbox{Ker} \: f^A_0
\longrightarrow \mbox{Ker} \: f^B_0$  such that the following
diagram with exact rows is commutative:

  \begin{picture}(80,60)
      \put(75,10){
      \put(0,39){\HVCenter{\footnotesize $0$}}
      \put(8,38){\vector(1,0){32}}
      \put(60,40){\HVCenter{\footnotesize $\mbox{Ker}\: f^A_0$}}
      \put(78,38){\vector(1,0){32}}
      \put(97,46){\HVCenter{\footnotesize $\psi^A$}}
      \put(122,40){\HVCenter{\small $\textsf{Q}^A_0$}}
      \put(133,38){\vector(1,0){32}}
      \put(150,46){\HVCenter{\footnotesize $f^A_0$}}
      \put(175,40){\HVCenter{\footnotesize $S_h$}}
      \put(186,38){\vector(1,0){32}}
      \put(225,39){\HVCenter{\footnotesize $0$}}

      \put(60,32){\vector(0,-1){25}}
      \put(53,20){\HVCenter{\footnotesize $\widetilde{h_0}$}}
      \put(122,32){\vector(0,-1){25}}
      \put(115,20){\HVCenter{\footnotesize $h_0$}}
      \put(175,32){\vector(0,-1){25}}
      \put(168,20){\HVCenter{\footnotesize $Id$}}

      \put(0,-1){\HVCenter{\footnotesize $0$}}
      \put(8,-2){\vector(1,0){32}}
      \put(60,0){\HVCenter{\footnotesize $\mbox{Ker}\: f^B_0$}}
      \put(78,-2){\vector(1,0){32}}
      \put(97,6){\HVCenter{\footnotesize $\psi^B$}}
      \put(122,0){\HVCenter{\small $\textsf{Q}^B_0$}}
      \put(133,-2){\vector(1,0){32}}
      \put(150,6){\HVCenter{\footnotesize $f^B_0$}}
      \put(175,0){\HVCenter{\footnotesize $S_h$}}
      \put(186,-2){\vector(1,0){32}}
      \put(225,-2){\HVCenter{\footnotesize $0$}}
     }
  \end{picture}

\noindent  Furthermore, it is clear that   $\widetilde{h_0}$ is an
epimorphism.

Consider the diagram

  \begin{picture}(80,170)
      \put(130,130){
      \put(0,40){\HVCenter{\small $\textsf{Q}^A_1$}}
      \put(120,40){\HVCenter{\small $\textsf{Q}^A_0$}}
      \put(60,0){\HVCenter{\footnotesize $\mbox{Ker}\: f^A_0$}}
      \put(10,38){\vector(1,0){100}}
      \put(60,30){\HVCenter{\footnotesize $f^A_1$}}
      \put(10,30){\vector(4,-3){32}}
      \put(17,15){\HVCenter{\footnotesize $\varphi^A$}}
      \put(75,8){\vector(4,3){32}}
      \put(98,15){\HVCenter{\footnotesize $\psi^A$}}

      \put(0,-60){\HVCenter{\small $\textsf{Q}^B_1$}}
      \put(120,-60){\HVCenter{\small $\textsf{Q}^B_0$}}
      \put(60,-100){\HVCenter{\footnotesize $\mbox{Ker}\: f^B_0$}}
      \put(10,-62){\vector(1,0){100}}
      \put(45,-50){\HVCenter{\footnotesize $f^B_1$}}
      \put(10,-70){\vector(4,-3){32}}
      \put(17,-85){\HVCenter{\footnotesize $\varphi^B$}}
      \put(75,-92){\vector(4,3){32}}
      \put(98,-85){\HVCenter{\footnotesize $\psi^B$}}

      \put(90,-128){\HVCenter{\footnotesize $0$}}
      \put(70,-110){\vector(1,-1){15}}

      \put(120,30){\vector(0,-1){82}}
      \put(130,-10){\HVCenter{\footnotesize $h_0$}}
      \put(65,-10){\vector(0,-1){82}}
      \put(75,-25){\HVCenter{\footnotesize $\widetilde{h_0}$}}

             \qbezier[30](0,30)(0,-11)(0,-52)
      \put(0,-51.7){\vector(0,-1){.1}}
      \put(-10,-10){\HVCenter{\footnotesize $h_1$}}
      }
  \end{picture}

For the epimorphism  $\varphi^B : \textsf{Q}^B_1 \longrightarrow
\mbox{Ker}\: f^B_0$, there exists a morphism of $A$-modules $h_1 :
\textsf{Q}^A_1 \longrightarrow \textsf{Q}^B_1$ (since
$\textsf{Q}^A_1$ is projective module) such that $\varphi^B h_1 =
\widetilde{h_0} \varphi^A$ and $ f^B_1 h_1 = \psi^B \varphi^B h_1
= \psi^B \widetilde{h_0} \varphi^A = h_0 \psi^A \varphi^A = h_0
f^A_1$.

Repeating the arguments, we get the following commutative diagram
with exact rows:

  \begin{picture}(80,160)
      \put(49,105){
      \put(0,40){\HVCenter{\small $\textsf{Q}^A_s$}}
      \put(120,40){\HVCenter{\small $\textsf{Q}^A_{s-1}$}}
      \put(60,0){\HVCenter{\footnotesize $\mbox{Ker}\: f^A_{s-1}$}}
      \put(10,38){\vector(1,0){96}}
      \put(10,30){\vector(4,-3){32}}
      \put(75,8){\vector(4,3){32}}

      \put(0,-60){\HVCenter{\small $\textsf{Q}^B_s$}}
      \put(120,-60){\HVCenter{\small $\textsf{Q}^B_{s-1}$}}
      \put(60,-100){\HVCenter{\footnotesize $\mbox{Ker}\: f^B_{s-1}$}}
      \put(10,-62){\vector(1,0){96}}
      \put(45,-50){\HVCenter{\footnotesize $f^B_s$}}
      \put(10,-70){\vector(4,-3){32}}
      \put(17,-85){\HVCenter{\footnotesize $\varphi^B_s$}}
      \put(75,-92){\vector(4,3){32}}
      \put(98,-85){\HVCenter{\footnotesize $\psi^B_s$}}

      \put(120,30){\vector(0,-1){82}}
      \put(132,-10){\HVCenter{\footnotesize $h_{s-1}$}}
      \put(65,-10){\vector(0,-1){82}}
      \put(78,-28){\HVCenter{\footnotesize $\widetilde{h_{s-1}}$}}

      \put(0,30){\vector(0,-1){82}}
      \put(10,-10){\HVCenter{\footnotesize $h_s$}}

      \put(132,38){\vector(1,0){20}}
             \qbezier[2](158,38)(161,38)(164,38)
      \put(170,38){\vector(1,0){20}}
      \put(200,40){\HVCenter{\small $\textsf{Q}^A_1$}}
      \put(208,38){\vector(1,0){20}}
      \put(239,40){\HVCenter{\small $\textsf{Q}^A_0$}}
      \put(247,38){\vector(1,0){20}}
      \put(256,50){\HVCenter{\footnotesize $f^A_1$}}
      \put(275,40){\HVCenter{\footnotesize $S_h$}}
      \put(282,38){\vector(1,0){20}}
      \put(308,39){\HVCenter{\footnotesize $0$}}

      \put(132,-62){\vector(1,0){20}}
             \qbezier[2](158,-62)(161,-62)(164,-62)
      \put(170,-62){\vector(1,0){20}}
      \put(200,-60){\HVCenter{\small $\textsf{Q}^B_1$}}
      \put(208,-62){\vector(1,0){20}}
      \put(239,-60){\HVCenter{\small $\textsf{Q}^B_0$}}
      \put(247,-62){\vector(1,0){20}}
      \put(256,-50){\HVCenter{\footnotesize $f^B_1$}}
      \put(275,-60){\HVCenter{\footnotesize $S_h$}}
      \put(282,-62){\vector(1,0){20}}
      \put(308,-61){\HVCenter{\footnotesize $0$}}

      \put(198,30){\vector(0,-1){82}}
      \put(208,-10){\HVCenter{\footnotesize $h_1$}}
      \put(239,30){\vector(0,-1){82}}
      \put(246,-10){\HVCenter{\footnotesize $h_0$}}
      \put(275,30){\vector(0,-1){82}}
      \put(281,-10){\HVCenter{\footnotesize $Id$}}

      \put(-30,38){\vector(1,0){20}}
      \put(-35,39){\HVCenter{\footnotesize $0$}}
      \put(-27,-62){\vector(1,0){16}}
      \put(-37,-60){\HVCenter{\footnotesize $P^B_{s+1}$}}
      \put(-35,30){\vector(0,-1){82}}
      \put(-22,-10){\HVCenter{\footnotesize $h_{s+1}$}}
      }
  \end{picture}

\noindent  where  $h_i$ are epimorphisms. Therefore,
$\textsf{Q}^B_{s+1} = 0$ and  $pd_B \: S_h \leq pd_A \: S_h \leq
n-1$. Then the algebra $B$ has global dimension equal to $n-1$,
since $pd_B \: \mbox{rad}\: \textsf{Q}_0 = n-1$. From Proposition
\ref{prop:dpM 2}, it follows that there is $S_{j_1}$, a
composition factor of $\mbox{rad}\: \textsf{Q}_0$. such that $pd_B
\: S_{j_1} = n-1$. Consequently, $pd_A \: S_{j_1} = n-1$.

For $1\leq i \leq n-1$, consider  $J_i =\{ h \in Q_0 : pd_A \: S
\geq n-i \}$  and $B_i = A / \langle\{e_h : h \in J_i \}\rangle$.
A similar argument to the one used previously shows that
$\mbox{gl.dim.}\: B_i = n-(i+1)$. Therefore, there is a simple
$A$-module $S_{j_{i+1}}$ such that $pd_A \: S_{j_{i+1}} =
n-(i+1)$, and the proof is complete.
\end{proof}

Note that the dual  statements of Propositions \ref{prop:lemaA},
\ref{prop:lemaB}, \ref{prop:dpM 2}, Lemmas \ref{lem:afirm 1},
\ref{lem:afirm 2} and Theorem \ref{thm:reduccion de la dim gl}
hold.

\section{Critical algebras.}

In this Section we  introduce a new family of algebras of global
dimension three, the critical algebras. We characterize these
algebras by quivers with relations. Finally, we show that every
strongly simply  connected schurian algebra, having global
dimension at least three, must contain a critical algebra.

\begin{defn} Let  $B$ be an  algebra. We say that  $B$ is
\textit{critical} if it satisfies the following properties:

\begin{enumerate}
    \item[i)] $B$ has a unique source $i$ and a unique sink  $j$,
    \item[ii)] $pd\: S_i = id\: S_j = 3$ and if $S$ is a different simple $B$-module we have that $pd\:  S \leq 2 $ and $id\: S \leq
    2$,
    \item[iii)] Let $0 \rightarrow \textsf{Q}_3 \rightarrow \textsf{Q}_2
    \rightarrow \textsf{Q}_1 \rightarrow \textsf{Q}_0 \rightarrow S_i
    \rightarrow 0$ be the minimal projective resolution  of the simple $S_i$ and let $\textsf{Q} = \oplus^{3}_{k=0} \textsf{Q}_k$.
    Then all  indecomposable projective $B$-modules are in $\mathrm{add}$$\: \textsf{Q}$, and each  indecomposable projective $B$-module
    is a direct summand of exactly one $\textsf{Q}_k$,
    \item[iv)] Let $0 \rightarrow S_j \rightarrow I_0 \rightarrow I_1 \rightarrow I_2 \rightarrow I_3
    \rightarrow 0$ be the minimal co-injective resolution  of simple  $S_j$ and let $I = \oplus^{3}_{k=0} I_k$. Then all
    indecomposable injective $B$-modules are in $\mathrm{add}$$\: I$, and each indecomposable injective $B$-module
    is a direct summand of exactly one $I_k$.

    \item[v)] $B$ does not contain any proper full subcategory that verifies i), ii), iii) and iv).
\end{enumerate}
\end{defn}

Note that a critical algebra $B$ has global dimension three. In
addition, $B$ is minimal in the following sense: $B$ does not
contain any proper full subcategory $B$ whose global dimension is
three.

\begin{thm}\label{thm:lema 7}
Let $A = kQ / I $ be a strongly simply connected schurian algebra.
Let $i, j$ be vertices of $Q$ and let

\begin{picture}(80,25)
      \put(130,10){
      \put(-92,-2){\HVCenter{\footnotesize $\cdots$}}
      \put(-82,-2){\vector(1,0){32}}
      \put(-38,0){\HVCenter{\small $\textsf{Q}_3$}}
      \put(-30,-2){\vector(1,0){32}}
      \put(12,0){\HVCenter{\small $\textsf{Q}_2$}}
      \put(22,-2){\vector(1,0){32}}
       \put(65,0){\HVCenter{\small $\textsf{Q}_1$}}
      \put(74,-2){\vector(1,0){32}}
       \put(116,0){\HVCenter{\footnotesize $P_i$}}
      \put(125,-2){\vector(1,0){32}}
       \put(165,0){\HVCenter{\footnotesize $S_i$}}
      \put(172,-2){\vector(1,0){32}}
      \put(210,-2){\HVCenter{\footnotesize $0$}}
     }
  \end{picture}

\noindent be the minimal projective resolution of  $S_i$ in $A$,
and let $P_j$ be a direct summand of  $\textsf{Q}_3$. If $C =
\mathrm{Conv}$$(i,j)$, then there exists a projective $C$-module
 $P$ such that $\Gamma = \mathrm{End}$$_C \: P$ is
critical.
\end{thm}

\begin{proof} Let

  \begin{picture}(80,25)
      \put(147,10){
      \put(-142,-1){\HVCenter{\footnotesize $0$}}
      \put(-137,-2){\vector(1,0){32}}
      \put(-94,0){\HVCenter{\footnotesize $P^{\alpha}_j$}}
      \put(-85,-2){\vector(1,0){32}}
      \put(-27,0){\HVCenter{\footnotesize $\bigoplus^t_{k=1} P^{\alpha_k}_{b_k}$}}
      \put(-2,-2){\vector(1,0){32}}
      \put(55,0){\HVCenter{\footnotesize $\bigoplus^u_{k=1} P_{a_k}$}}
      \put(78,-2){\vector(1,0){32}}
      \put(118,0){\HVCenter{\footnotesize $P_i$}}
      \put(125,-2){\vector(1,0){32}}
      \put(165,0){\HVCenter{\footnotesize $S_i$}}
      \put(172,-2){\vector(1,0){32}}
      \put(210,-2){\HVCenter{\footnotesize $0$}}
     }
  \end{picture}

\noindent be the minimal projective resolution of $S_i$ in $C$.
Consider

\begin{itemize}
\item $\mathcal{R }$ $= \{b \in Q_0 : P_b $ is a direct summand of
$\textsf{Q}_2$ and there is a nonzero path $b \leadsto j \}$,

 \item $\mathcal{S}$ $= \{a \in Q_0 : P_a $ is a direct summand of $\textsf{Q}_1$, there are $b \in \mathcal{R}$ and $\rho \in
I$ such that  $a \leadsto b$ is a nonzero path and  $\rho: a
\leadsto j \} = \{ a_1, \ldots, a_s \}$ ,
\end{itemize}

Assume that, for  $1 \leq k \leq v$, $\rho_k: a_k \leadsto j$,
there are commutativity relations going through  $b_k, b'_k$, and
that $v+1 \leq k \leq s$, $\rho_k$ are monomial relations that go
through $b_k$.

Consider the projective $C$-module $$P = P_i \oplus \bigoplus_{a
\in \mathcal{S}} P_{a} \oplus \bigoplus_{b \in \mathcal{R}} P_{b}
\oplus P_j$$

By abuse of language, we also call  $i, a_1, \ldots , a_s, b_1,
\ldots , b_r, , j$ the vertices of   $\Gamma$, where  $r =
\mbox{Card}\: \mathcal{R }$.

Then in $\Gamma= kQ_{\Gamma}/ I_{\Gamma}$ we have:

\begin{enumerate}
\item[$\star_1.$] $(Q_{\Gamma})_0 = \{i, a_1, \ldots , a_s,
b_1, \ldots , b_r, , j\}$

\item[$\star_2.$] there is a only source $i$ and only sink $j$

\item[$\star_3.$] arrows starting at $i$ are $i \longrightarrow a_k, \  1\leq k \leq s$

\item[$\star_4.$] arrows ending at $j$ are $b_h \longrightarrow j, \  1\leq
h \leq r$

\item[$\star_5.$] there are arrows  $a_k \longrightarrow b_h$, for $1\leq k \leq
s, 1 \leq h \leq r$, where there are nonzero  paths $a_k \leadsto
b_h$ in the algebra $C$

\item[$\star_6.$] $(Q_{\Gamma})_1$ consists of all the arrows mentioned in $\star_3.$, $\star_4.$ and $\star_5.$.

\item[$\star_7.$]  for each $1 \leq h \leq r$, there is a relation $i \leadsto
b_h$ (because this relation exists in $C$)

\item[$\star_8.$] for each $1 \leq k \leq s$, there is a relation $a_k
\leadsto j$ (because this relation exists in $C$)

\item[$\star_9.$] $I_{\Gamma}$ is generated by the relations given in
$\star_7.$ and $\star_9.$

\item[$\star_{10}.$] $pd_{\Gamma}\: S_j = 0$ \  and  \  $id_{\Gamma}\: S_i =
0$, \  since $j$ is the sink and  $i$ is the source in $\Gamma$.

\item[$\star_{11}.$] $pd_{\Gamma}\: S_{b_h} = 1$ for $1\leq h \leq r$
\  and  \  $id_{\Gamma}\: S_{a_k} = 1$ for $1 \leq k \leq s$,
since
 there are no relations starting at $b_h$  and no relations ending at $a_k$ in $\Gamma$.

\item[$\star_{12}.$] $pd_{\Gamma}\: S_{a_k} = 2$ for $1 \leq k \leq
s$ \   and \   $id_{\Gamma}\: S_{b_h} = 2$ for $1\leq h \leq r$,
since there are minimal relations starting at $ a_k $,  and there
are minimal relations ending at $b_h$ in  $\Gamma$.
\end{enumerate}

Then a projective resolution of $S_i$ in $\Gamma$ is of the form

  \begin{picture}(80,20)
      \put(140,10){
     \put(-130,-1){\HVCenter{\footnotesize $0$}}
      \put(-125,-2){\vector(1,0){22}}
      \put(-92,0){\HVCenter{\small $\textsf{Q}^{\Gamma}_3$}}
      \put(-84,-2){\vector(1,0){22}}
      \put(-10,0){\HVCenter{\footnotesize $\bigoplus^s_{k=1} P_{b_k} \oplus \bigoplus^v_{k=1} P_{b'_k}$}}
      \put(38,-2){\vector(1,0){22}}
      \put(48,6){\HVCenter{\footnotesize $f_{2}$}}
      \put(85,0){\HVCenter{\footnotesize $\bigoplus^s_{k=1} P_{a_k}$}}
      \put(108,-2){\vector(1,0){22}}
      \put(118,6){\HVCenter{\footnotesize $f_1$}}
      \put(138,0){\HVCenter{\footnotesize $P_i$}}
      \put(145,-2){\vector(1,0){22}}
      \put(155,6){\HVCenter{\footnotesize $f_0$}}
      \put(175,0){\HVCenter{\footnotesize $S_i$}}
      \put(182,-2){\vector(1,0){22}}
      \put(210,-2){\HVCenter{\footnotesize $0$}}
     }
  \end{picture}

Since $C$ is triangular, so is $\Gamma$. Consequently, the
projective modules that occur as direct summands of the terms
$\textsf{Q}^{\Gamma}_0 = P_i$, \ $\textsf{Q}^{\Gamma}_1 =
\bigoplus^s_{k=1} P_{a_k}$, \ $\textsf{Q}^{\Gamma}_2 =
\bigoplus^s_{k=1} P^{\alpha_k}_{b_k} \oplus \bigoplus^v_{k=1}
P^{\alpha'_k}_{b'_k}$ can  not appear as direct summands of
$\textsf{Q}^{\Gamma}_3$. Then either $\textsf{Q}^{\Gamma}_3 = 0$
or $\textsf{Q}^{\Gamma}_3 = P^{\alpha}_j$.

We show that $\textsf{Q}^{\Gamma}_3 \neq 0$. In fact, since
$\mu_{\textsf{Q}^{\Gamma}_1} (S_j) =  v$ and
$\mu_{\textsf{Q}^{\Gamma}_2} (S_j) = s + v$, then
$\mu_{{\scriptsize \mbox{Ker}} f_2} (S_j) \geq s
> 0$. Since the relations $\rho_k :a_k \leadsto j$ are minimal relations
in $\Gamma$, for each $1 \leq k \leq s$, it follows that  $S_j \in
\mbox{Top Ker}\: f_2$. Therefore,  $pd_{\Gamma}\: S_i = 3$, and
hence $\mbox{gl.dim.}\: \Gamma = 3$ and $id_{\Gamma}\: S_j = 3$.
Then $\Gamma$ is critical. \end{proof}

We give below a description by quivers with relations of all
critical algebras.

\begin{prop}\label{prop:qv de *3}
Let $\Gamma$ be a critical algebra. Then either $\Gamma$ or
$\Gamma^{op}$ is one of the following algebras.

  
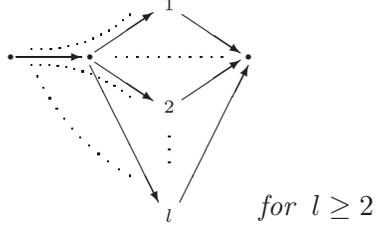
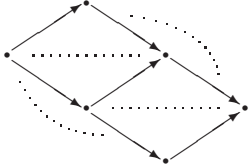
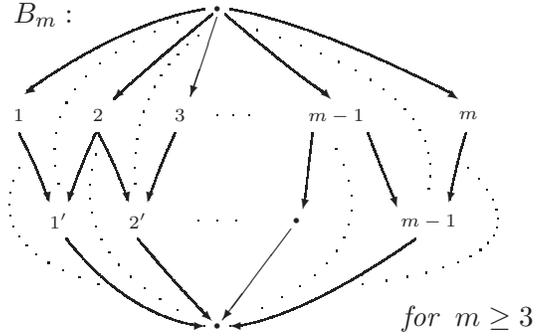
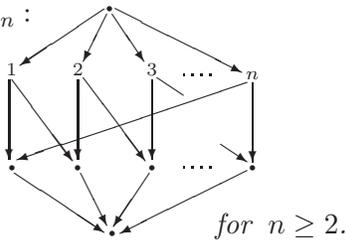
\begin{picture}(10,100)

   \put(5,65){
             \put(0,25){$A_1:$}
      \put(0,0){\circle*{2}}
      \put(30,0){\circle*{2}}
      \put(60,0){\circle*{2}}
      \put(90,0){\circle*{2}}
      \put(2,0){\vector(1,0){26}}
      \put(32,0){\vector(1,0){26}}
      \put(62,0){\vector(1,0){26}}
            \qbezier[16](10,2)(30,20)(45,2)
            \qbezier[16](38,2)(53,20)(72,2)
    }

   \put(215,65){
             \put(0,25){$A_l:$}
             \put(90,-60){ $\mbox{\textit{for }} \: l \geq 2$}
      \put(0,0){\circle*{2}}
      \put(30,0){\circle*{2}}
      \put(60,-19){\HVCenter{\tiny $2$}}
      \put(60,20){\HVCenter{\tiny $1$}}
      \put(90,0){\circle*{2}}
      \put(60,-60){\HVCenter{\tiny $l$}}
      \put(2,0){\vector(1,0){25}}
      \put(32,2){\vector(3,2){22}}
      \put(32,-2){\vector(3,-2){22}}
      \put(30,-3){\vector(1,-2){26}}
      \put(65,16){\vector(3,-2){22}}
      \put(65,-16){\vector(3,2){22}}
      \put(64,-55){\vector(1,2){26}}
            \qbezier[2](60,-40)(60,-35)(60,-30)
            \qbezier[12](8,3)(30,2)(45,16)
            \qbezier[10](40,0)(60,0)(80,0)
            \qbezier[12](8,-3)(28,-2)(45,-15)
            \qbezier[12](8,-3)(20,-30)(45,-45)
    }
\end{picture}

  \begin{picture}(10,130)
   \put(3,70){
             \put(-2,45){$B_1:$}
      \put(0,20){\circle*{2}}
      \put(30,40){\circle*{2}}
      \put(30,0){\circle*{2}}
      \put(60,20){\circle*{2}}
      \put(60,-20){\circle*{2}}
      \put(90,0){\circle*{2}}
      \put(2,22){\vector(3,2){26}}
      \put(2,18){\vector(3,-2){26}}
      \put(32,39){\vector(3,-2){26}}
      \put(32,1){\vector(3,2){26}}
      \put(62,19){\vector(3,-2){26}}
      \put(32,-2){\vector(3,-2){26}}
      \put(62,-19){\vector(3,2){26}}
            \qbezier[12](10,20)(30,20)(50,20)
            \qbezier[12](40,0)(60,0)(80,0)
            \qbezier[12](47,35)(77,30)(80,13)
            \qbezier[12](5,10)(15,-10)(36,-10)
    }

   \put(155,40){
             \put(-2,75){$B_m:$}
             \put(145,-40){$\mbox{\textit{for }} \: m \geq 3$}
      \put(75,80){\circle*{2}}
      \put(0,40){\HVCenter{\tiny $1$}}
      \put(30,40){\HVCenter{\tiny $2$}}
      \put(61,40){\HVCenter{\tiny $3$}}
            \qbezier[2](75,40)(81,40)(87,40)
      \put(120,40){\HVCenter{\tiny $m-1$}}
      \put(170,40){\HVCenter{\tiny $m$}}
      \put(15,0){\HVCenter{\tiny $1'$}}
      \put(45,0){\HVCenter{\tiny $2'$}}
            \qbezier[2](68,0)(75,0)(82,0)
      \put(105,0){\circle*{2}}
      \put(155,0){\HVCenter{\tiny $m-1$}}
      \put(75,-40){\circle*{2}}

            \qbezier(3,49)(40,75)(70,80)
      \put(2.2,48){\vector(-3,-2){0.1}}
            \qbezier(36,46)(45,55)(73,78)
      \put(36,46){\vector(-1,-1){0.1}}
      \put(75,77){\vector(-1,-3){10}}
            \qbezier(116,48)(105,55)(78,78)
      \put(118,46){\vector(1,-1){0.1}}
            \qbezier(164,48)(110,75)(80,80)
      \put(166.2,46){\vector(3,-2){0.1}}

            \qbezier(0,33)(7,20)(11,8)
      \put(11.5,7){\vector(1,-4){0.1}}
            \qbezier(30,33)(37,20)(41,8)
      \put(41.5,7){\vector(1,-4){0.1}}
            \qbezier(132,33)(137,20)(142,8)
      \put(143,5){\vector(1,-3){0.1}}

            \qbezier(29,33)(23,20)(19,8)
      \put(18.7,7){\vector(-1,-4){0.1}}
            \qbezier(59,33)(53,20)(49,8)
      \put(48.7,7){\vector(-1,-4){0.1}}

            \qbezier(111,33)(109,15)(108,7)
      \put(107.5,4.5){\vector(-1,-4){0.1}}

            \qbezier(169,33)(165,20)(163,8)
      \put(162.7,7){\vector(-1,-4){0.1}}

            \qbezier(18,-7)(40,-35)(68,-40)
      \put(70.2,-40){\vector(1,0){0.1}}
            \qbezier(45,-7)(60,-25)(72,-37)
      \put(73,-38){\vector(1,-1){0.1}}
      \put(103,-3){\vector(-3,-4){26}}
            \qbezier(150,-7)(110,-35)(83,-40)
      \put(80,-40){\vector(-1,0){0.1}}

            \qbezier[12](60,75)(5,55)(15,14)
            \qbezier[12](68,70)(42,50)(44,14)
            \qbezier[16](90,72)(162,50)(155,12)

            \qbezier[12](29,25)(18,-10)(60,-34)
            \qbezier[12](120,27)(140,-10)(92,-33)

            \qbezier[12](0,20)(-15,-10)(30,-24)
            \qbezier[16](170,20)(205,-15)(130,-24)

    }

\end{picture}

  \begin{picture}(50,95)
         \put(100,5){
                \put(0,75){$Q_n:$}
                \put(90,-5){$\mbox{\textit{for }} \: n \geq 2$.}
      \put(50,80){\circle*{2}}
      \put(66,57){\HVCenter{\tiny $3$}}
      \put(38,57){\HVCenter{\tiny $2$}}
      \put(13,57){\HVCenter{\tiny $1$}}
      \put(104,55){\HVCenter{\tiny $n$}}
      \put(66,20){\circle*{2}}
      \put(104,20){\circle*{2}}
      \put(13,20){\circle*{2}}
      \put(38,20){\circle*{2}}
      \put(51,-5){\circle*{2}}
      \put(48,80){\vector(-3,-2){32}}
      \put(49,78){\vector(-1,-2){9}}
      \put(51,78){\vector(2,-3){12}}
      \put(52,80){\vector(2,-1){48}}
      \put(66,53){\vector(0,-1){30}}
      \put(38,53){\vector(0,-1){30}}
      \put(12,53){\vector(0,-1){30}}
      \put(104,52){\vector(0,-1){30}}
      \put(102,52){\vector(-3,-1){87}}
      \put(14,18){\vector(3,-2){34}}
      \put(39,18){\vector(1,-2){10}}
            \put(68,54){\line(3,-2){10}}
            \put(102,22){\line(-3,2){10}}
            \put(102,22){\vector(3,-2){0.10}}
      \put(65,18){\vector(-2,-3){13.5}}
      \put(102,19){\vector(-2,-1){48}}
      \put(13,53){\vector(3,-4){23}}
      \put(40,54){\vector(3,-4){23.5}}
             \qbezier[3](78,55)(83,55)(88,55)
             \qbezier[3](78,20)(83,20)(88,20)
    }
\end{picture}

\end{prop}

\begin{proof} Let $\Gamma = kQ/I$ be a critical algebra.
From the proof of the Theorem \ref{thm:lema 7} and the properties
$iii)$ and $iv)$ of the critical algebras, it follows that there
is the following partition in the set of vertices of  $\Gamma$:
$$Q_0=\{ i \} \cup\{a_1, \ldots , a_s \} \cup \{b_1, \ldots, b_r\}
\cup \{ j \}$$ such that:
\begin{itemize}
\item there are $i \longrightarrow a_k \in Q_0$, for all $1 \leq k \leq
s$,
\item there are arrows of the form $a_k \longrightarrow b_h$,
\item there are $b_h \longrightarrow j \in Q_0$, for all $1 \leq h \leq
r$,
\item for each $1 \leq h \leq r$, there is a relation $i \leadsto
b_h$.
\item for each $1 \leq k \leq s$, there is a relation $a_k
\leadsto j$,
\end{itemize}

Then, variating $s$ and $r$, and using the property $v)$, we get
all the quivers with relations for $\Gamma$. These quivers are the
ones listed above, or their oppossite quivers. \end{proof}

We are now able to state our main theorem.

\begin{thm}\label{thm:ida}
Let $A = kQ / I$ be a strongly simply connected schurian algebra
with $\mathrm{gl.dim.}$$\: A \geq 3$. Then there exists a full
subcategory  $B$ of $A$ such that $B$ is critical.
\end{thm}

\begin{proof} Since $\mathrm{gl.dim.}$$\: A \geq 3$, by
Theorem \ref{thm:reduccion de la dim gl}, there exists a vertex $i
\in Q_0$ with $dp\: S_i = 3$ and a vertex $j \in Q_0$ such that
$P_j$ is a direct summand of the term $\textsf{Q}_3$ of the
minimal projective resolution  of $S_i$. Consider
$\mbox{Conv}(i,j)$ and apply Lemma  \ref{lem:pozo-fuente} and
Theorem \ref{thm:lema 7} to obtain the desired result.
\end{proof}

\begin{exmp}
Let $A$ be the algebra given by the following quiver with
relations

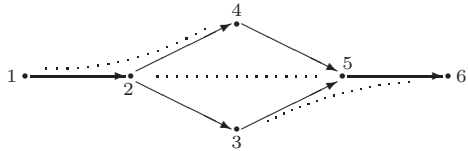
\begin{picture}(10,57)

   \put(115,27){
      \put(0,0){\circle*{2}}
             \put(-3,0){\RVCenter{\tiny $1$}}
      \put(40,0){\circle*{2}}
             \put(41,-5){\RVCenter{\tiny $2$}}
      \put(80,-20){\circle*{2}}
             \put(82,-25){\RVCenter{\tiny $3$}}
      \put(80,20){\circle*{2}}
             \put(82,25){\RVCenter{\tiny $4$}}
      \put(120,0){\circle*{2}}
             \put(124,5){\RVCenter{\tiny $5$}}
      \put(160,0){\circle*{2}}
             \put(167,0){\RVCenter{\tiny $6$}}
      \put(2,0){\vector(1,0){36}}
      \put(42,2){\vector(2,1){36}}
      \put(42,-2){\vector(2,-1){36}}
      \put(82,20){\vector(2,-1){36}}
      \put(82,-20){\vector(2,1){36}}
      \put(122,0){\vector(1,0){36}}
            \qbezier[16](8,3)(40,2)(68,18)
            \qbezier[16](50,0)(80,0)(110,0)
            \qbezier[16](145,-2)(115,-5)(92,-17)
    }

\end{picture}

\noindent This algebra is strongly simply connected and schurian
with global dimension three. Note that $A$ contains a critical
algebra. In fact, if we consider  $B = \mathrm{End}$$_A (P_1
\oplus P_2 \oplus P_5 \oplus P_6)$, we get that $B$ is a critical
algebra,  which is isomorphic to $A_1$.

\end{exmp}

It follows from Theorem \ref{thm:ida}, that if a strongly simply
connected schurian algebra  $A$ does not contain a critical full
subcategory, then  $\mbox{gl.dim}\: A \leq 2$. Therefore, we give
a sufficient condition for deciding if the algebra has global
dimension two.

The following example shows that the converse of Theorem
\ref{thm:ida} does not hold in general. An algebra of global
dimension two may
 have a critical full subcategory.

\begin{exmp}
Consider the algebra  $A=kQ/I$, given by the following quiver with
relations

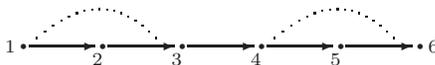
\begin{picture}(10,30)

   \put(125,10){
      \put(0,0){\circle*{2}}
             \put(-3,0){\RVCenter{\tiny $1$}}
      \put(30,0){\circle*{2}}
             \put(30,-5){\RVCenter{\tiny $2$}}
      \put(60,0){\circle*{2}}
             \put(60,-5){\RVCenter{\tiny $3$}}
      \put(90,0){\circle*{2}}
             \put(90,-5){\RVCenter{\tiny $4$}}
      \put(120,0){\circle*{2}}
             \put(120,-5){\RVCenter{\tiny $5$}}
      \put(150,0){\circle*{2}}
             \put(157,0){\RVCenter{\tiny $6$}}
      \put(2,0){\vector(1,0){25}}
      \put(32,0){\vector(1,0){25}}
      \put(62,0){\vector(1,0){25}}
      \put(92,0){\vector(1,0){25}}
      \put(122,0){\vector(1,0){25}}
            \qbezier[16](5,3)(30,25)(50,3)
            \qbezier[16](95,3)(120,25)(140,3)
    }

\end{picture}

The algebra $A$ is  strongly simply connected and schurian, with
global dimension two. However, the subcategory  $B =
\mathrm{End}$$_A (P_1 \oplus P_2 \oplus P_5 \oplus P_6)$ is
critical, isomorphic to $A_1$.
\end{exmp}


\end{document}